\documentclass[letterpaper,oneside]{amsart}
\usepackage{times}
\usepackage[margin=3 cm,letterpaper]{geometry}
\usepackage[stable]{footmisc}
\usepackage{graphicx}
\usepackage[shortlabels]{enumitem}
\usepackage{microtype}
\usepackage{cite}
\usepackage{amssymb,amsmath,amsthm,mathrsfs}
\usepackage{hyperref}
\usepackage[capitalize]{cleveref}
\usepackage{tikz}
\usetikzlibrary{arrows,backgrounds,fit}
\usepackage[all]{xy}
\usepackage{dsfont} 
\usepackage[euler]{textgreek}
\usepackage[foot]{amsaddr}
\usepackage[misc,geometry]{ifsym}
\usepackage{xspace}
\usepackage{xcolor}
\usepackage[nomargin,inline,draft]{fixme}
\fxsetup{theme=color, mode=multiuser}
\FXRegisterAuthor{sa}{SA}{\color{red}Sameer}
\FXRegisterAuthor{rt}{RT}{\color{blue}Rekha}
\FXRegisterAuthor{dc}{DC}{\color{orange}Diego}

\newcommand{\NN}{{\mathbb N}}
\newcommand{\RR}{{\mathbb R}}
\newcommand{\PP}{{\mathbb P}}
\newcommand{\SR}{{\mathbb S}}
\DeclareMathOperator{\id}{\mathds{1}\mkern-1mu}

\DeclareMathOperator{\codim}{codim}
\DeclareMathOperator{\rank}{rank}
\DeclareMathOperator{\corank}{corank}

\DeclareMathOperator{\val}{val}
\DeclareMathOperator{\interior}{int}
\DeclareMathOperator{\closure}{cl}
\newcommand{\ACQ}[2]{{\mathrm{ACQ}_{#1}(#2)}} 

\renewcommand{\subset}{\subseteq}

\newcommand\Qtheta{{\eqref{eq:Qtheta}}\xspace}
\newcommand\Ptheta{{\eqref{eq:Ptheta}}\xspace}
\newcommand\Dtheta{{\eqref{eq:Dtheta}}\xspace}
\newcommand\Qthetabar{{(\hyperref[eq:Qtheta]{$Q_{\bar{\theta}}$})}\xspace}
\newcommand\Pthetabar{{(\hyperref[eq:Ptheta]{$P_{\bar{\theta}}$})}\xspace}
\newcommand\Dthetabar{{(\hyperref[eq:Dtheta]{$D_{\bar{\theta}}$})}\xspace}
\newcommand\Qthetaobj{{\eqref{eq:Qthetaobj}}\xspace}
\newcommand\Qthetabarobj{{(\hyperref[eq:Qthetaobj]{$Q_{\bar{\theta}}^{\mathrm{obj}}$})}\xspace}
\newcommand\Qthetatilde{{\eqref{eq:Qthetatilde}}\xspace}
\newcommand\Etheta{{\eqref{eq:nearestpoint}}\xspace}
\newcommand\Ethetabar{{(\hyperref[eq:nearestpoint]{$E_{\bar{\theta}}$})}\xspace}

\makeatletter
\DeclareRobustCommand{\textsupsub}[2]{{%
  \m@th\ensuremath{%
    ^{\mbox{\footnotesize#1}}%
    _{\mbox{\fontsize\sf@size\z@#2}}%
  }%
}}
\makeatother

\newtheorem{theorem}{Theorem}
\numberwithin{theorem}{section}
\newtheorem{proposition}[theorem]{Proposition}
\newtheorem{corollary}[theorem]{Corollary}
\newtheorem{lemma}[theorem]{Lemma}

\theoremstyle{definition}
\newtheorem{definition}[theorem]{Definition}
\newtheorem{example}[theorem]{Example}
\theoremstyle{remark}
\newtheorem{remark}[theorem]{Remark}

\let\oldtabular\tabular
\renewcommand{\tabular}{\footnotesize\oldtabular}
\title[On the local stability of SDP relaxations]{On the local stability of semidefinite relaxations}

\date{\today}
\author{Diego Cifuentes$^{\,1} (\textrm{\,\Letter})$}
\address{$\mkern-19mu^1$
Massachusetts Institute of Technology, Cambridge MA 02139, USA}
\email{diegcif@mit.edu}

\author{Sameer Agarwal$^{\,2}$}
\address{$^2$
Google Inc.
}
\email{sameeragarwal@google.com }

\author{Pablo A. Parrilo$^{\,1}$}
\email{parrilo@mit.edu}

\author{Rekha R. Thomas$^{\,3}$}
\address{$^3$
University of Washington, Seattle WA 98195, USA}
\email{rrthomas@uw.edu}


\begin{document}
\maketitle

\begin{abstract}
  We consider a parametric family of quadratically constrained quadratic programs (QCQP) and their associated semidefinite programming (SDP) relaxations.
  Given a nominal value of the parameter at which the SDP relaxation is exact,
  we study conditions (and quantitative bounds) under which the relaxation will continue to be exact as the parameter moves in a neighborhood around the nominal value.
  Our framework captures a wide array of statistical estimation problems including
  tensor principal component analysis, rotation synchronization, orthogonal Procrustes, camera triangulation and resectioning, essential matrix estimation, system identification, and  approximate~GCD.
  Our results can also be used to analyze the stability of SOS relaxations of general polynomial optimization problems.
\end{abstract}

\section{Introduction}
A number of parameter estimation problems \cite{Aholt2012,Wang2013,Fredriksson2012,Rosen2016,Nie2014,Chu2001,Cifuentes2018,Zhao2019,Kaltofen2006,Bookstein1979,Waldspurger2015,Biswas2004,Luo2010} in statistics and engineering
can be posed as minimizing a polynomial over an algebraic variety. For example, a commonly occurring form is:
\begin{align} \label{eq:prob2}
    \min_{y,z}\ \|y - \theta\|^2,
    \qquad\text{s.t.}\qquad
    g(y,z) = 0.
\end{align}
where $\theta$ is a vector of noisy observations from the {\em model} which is an algebraic variety described by a system of quadratic polynomials~$g$.
The aim is to find $y$ and $z$ that best explains the observations $\theta$, and
by {\em best} we mean the maximum likelihood estimate (MLE) of $y$ and~$z$%
\footnote{To keep the discussion simple,  we are assuming identical and independently distributed Gaussian noise, but many other choices are possible}.
Problem~\eqref{eq:prob2} is an instance of a quadratically constrained quadratic program (QCQP). While QCQPs are hard to solve, their Lagrangian duals are semidefinite programs (SDP) that can be solved efficiently.
In general this SDP is a relaxation of the QCQP
in the sense that its optimal value is only a lower bound to the
optimal value of the QCQP. However, in some instances, their values agree (i.e. the duality gap is zero)
and we say that the SDP relaxation is tight. Indeed this has been found to happen in a number of estimation problems in practice, when the noise is small. We illustrate this phenomenon on the following simple example:
\begin{example}[Nearest point to the twisted cubic]\label{exmp:twistedcubic}
  Let $\mathbf{Y} := \{(t,t^2,t^3): t\in \RR\}$ be the twisted cubic curve in~$\RR^3$.
  Consider the problem of finding the nearest point from $\theta\in \RR^3$ to the curve~$\mathbf{Y}$.
  This problem can be phrased as the QCQP:
  \begin{align}\label{eq:twistedcubic}
    \min_{y\in \mathbf{Y}}\quad &\|y-\theta\|^2, \quad \text{ where }\quad
    \mathbf{Y}= \{ y\in \RR^3: y_2 \!=\! y_1^2,\; y_3 \!=\! y_1y_2\}.
  \end{align}
  Its Lagrangian dual is the SDP:
  \begin{equation}\label{eq:twistedcubicdual}
  \begin{aligned}
    \max_{\lambda_0,\lambda_1,\lambda_2\in\RR}\quad & -\lambda_0,
    \quad\text{s.t.}\quad
    \left(\begin{smallmatrix}
      \lambda_0+ \|\theta\|^2 &-\theta_1   &\lambda_1-\theta_2  &\lambda_2-\theta_3 \\
      -\theta_1 &1-2\lambda_1  &-\lambda_2  &0 \\
      \lambda_1-\theta_2 &-\lambda_2  &1  &0 \\
      \lambda_2-\theta_3 &0  &0  &1
      \end{smallmatrix}\right) \succeq 0.
  \end{aligned}
  \end{equation}
   \Cref{fig:twistedcubic} shows the
  projection of $Y$ onto the $y_1y_3$-plane, and the duality gap for parameters~$\theta$ of the form $(\theta_1,\theta_1^2,\theta_3)$.
  For all parameters in the dotted region around $\mathbf{Y}$, the QCQP has
  zero-duality gap in the sense that its optimal value agrees with that of its Lagrangian dual.

  \begin{figure}[hbt]
    \centering
    \includegraphics[clip,trim=0 {8pt} 0 {20pt},width=370pt]{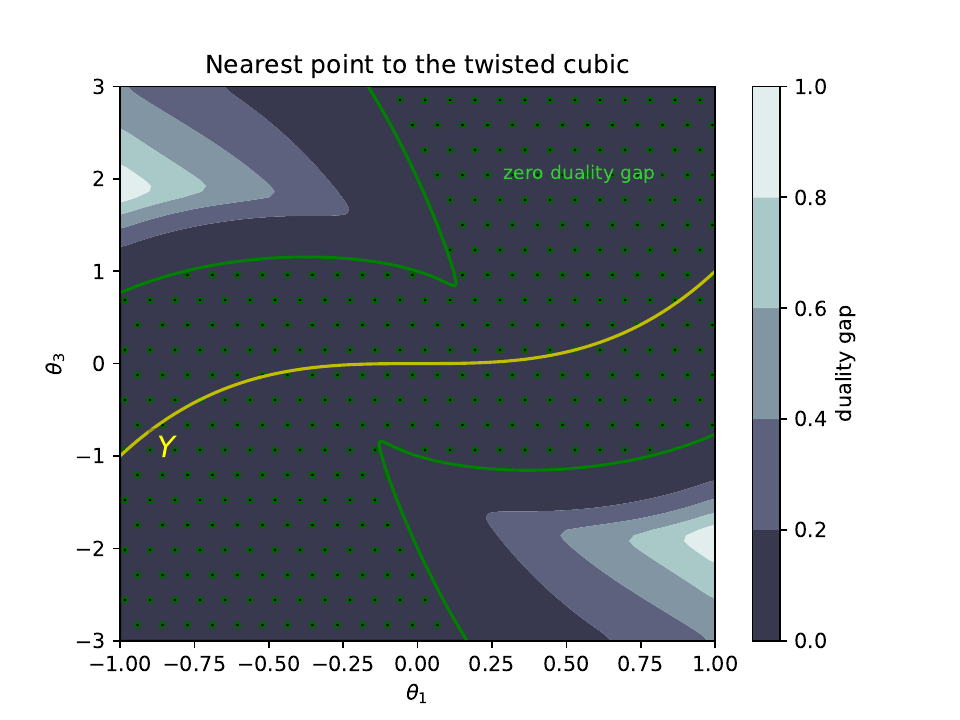}
    \caption{Duality gap in the QCQP~\eqref{eq:twistedcubic} for parameters $\theta$ of the form $(\theta_1,\theta_1^2,\theta_3)$.
  There is no duality gap in the dotted region.}
    \label{fig:twistedcubic}
  \end{figure}
\end{example}

Problem~\eqref{eq:prob2} has the property that if $\theta$ is {\em noiseless}, i.e., $\theta = \bar \theta$ lies on the algebraic variety,
then the QCQP corresponding to $\bar\theta$ has zero duality gap.
The interesting feature is that
in many instances, as in the above example, \eqref{eq:prob2} continues to have zero duality gap when $\theta$ is close to $\bar{\theta}$.
In some cases, there are problem specific explanations for when we can expect the relaxation to be tight under low noise~\cite{Aholt2012,Rosen2016,Wang2013,Fredriksson2012}.
There is, however, no general understanding of this phenomenon. The aim of this paper is to fill this gap.

One approach to studying the tightness of SDP relaxations in the
{\em low noise regime} is to think of problem~\eqref{eq:prob2} as a perturbation of a noiseless instance, whose relaxation is tight. This is the perspective we will take in this paper. Further, we will not limit ourselves to the special problem in~\eqref{eq:prob2} and will instead consider general QCQPs.

We point out that our focus on SDP relaxations from QCQPs does not pose a practical limitation,
since most of the SDP relaxations used in practice arise in such a way.
In some cases, though, the QCQP is not immediately apparent to the user.
For instance, the nuclear norm minimization problem
$\min\{\|M\|_* : M \!\in\! \RR^{m\times n}, \mathcal{A}(M) \!=\! b \}$,
which can be rewritten as an SDP,
and is often used~\cite{Fazel2002} to find a low rank matrix $M$ satisfying affine constraints $\mathcal{A}(M) \!=\! b$,
arises as the relaxation of the QCQP
$\min\{\frac{1}{2}(\|u\|^2 \!+\! \|v\|^2) : u \!\in\! \RR^m, v \!\in\! \RR^n, \mathcal{A}(u v^T) \!=\! b\}$.

The problem we consider is the following.
Let $\Theta\subset\RR^d$ be a parameter space,
and consider a family of QCQPs parametrized by $\theta\in \Theta$:
\begin{equation*} \label{eq:Qtheta}
\tag{\textit{{Q}\textsubscript{\texttheta}}}
\begin{aligned}
  \min_{x\in\RR^N}\quad
  & f_\theta(x) :=x^\top F_\theta x \\
  \text{s.t.}\quad
  & g_\theta^i(x):= x^\top G_\theta^i x + c^i_\theta =0,\quad
  i=1,\dots,m
\end{aligned}
\end{equation*}
where $F_\theta, G^i_\theta\in \SR^N$ are symmetric $N\times N$ matrices,
and $c^i_\theta\in \RR$ are scalars.
We assume that at least one $c^i_\theta$ is nonzero,
which ensures that $x\!=\!0$ is not a solution of \Qtheta.
We will also assume that the map $\theta\mapsto (F_\theta, G^i_\theta,c^i_\theta)$ is continuous.
A QCQP such as \Qtheta that involves no
linear terms is said to be {\em homogeneous}.
While our main results are formulated for
homogeneous QCQPs, in each case we also write down the
accompanying statement for the inhomogeneous version.

The QCQP~\Qtheta has an associated pair of SDP relaxations that are dual to each other:

\begin{minipage}{.50\linewidth}
\begin{equation*}\label{eq:Ptheta}
\tag{\textit{{P}\textsubscript{\texttheta}}}
\begin{aligned}
  \min_{S\in\SR^N}\quad
  & {F_\theta}\bullet S \\
  \text{s.t.}\quad
  & G^i_\theta \bullet S + c^i_\theta = 0, \quad i=1,\dots,m\\
  & S \succeq 0
\end{aligned}
\end{equation*}
\end{minipage}\qquad
\begin{minipage}{.40\linewidth}
\begin{equation*}\label{eq:Dtheta}
\tag{\textit{{D}\textsubscript{\texttheta}}}
\begin{aligned}
  \max_{\lambda\in \RR^m}\quad
  & \textstyle\sum_i \lambda_i c^i_\theta \\
  \text{s.t.}\quad
  & \mathcal{H}_\theta(\lambda)
  := F_\theta + \textstyle\sum_i \lambda_i G^i_\theta
  \succeq 0
\end{aligned}
\end{equation*}
\end{minipage}

\noindent
where $\bullet$ denotes the trace inner product in $\SR^N$.
The SDP \Ptheta will be referred to as the \emph{primal
SDP relaxation} of \Qtheta.
There is a bijection between the solutions of \Qtheta and the rank one solutions of \Ptheta via $x \mapsto xx^\top$.
Hence, if an optimal solution of \Ptheta has rank one, then
\Ptheta recovers a minimizer of \Qtheta.
Problem \Ptheta is sometimes known as the {\em Shor relaxation} of~\Qtheta.

Problem \Dtheta is the dual SDP of \Ptheta
as well as the {\em Lagrangian dual} of~\Qtheta. Its objective function is traditionally
written as
$\max_{\lambda \in \RR^m} \min_{x \in \RR^N} L_\theta(x,\lambda)$ where
\begin{align*}
    L_\theta(x,\lambda) = f_\theta(x) + \sum_{i} \lambda_i g^i_\theta(x)
\end{align*}
is the {\em Lagrangian function} of \Qtheta.
Note that
$ \mathcal{H}_\theta(\lambda) = \tfrac{1}{2}\, \nabla^2_{xx} L_\theta(x,\lambda)$
is half the Hessian of the Lagrangian.

The inequalities $\val\Qtheta \geq
\val\Ptheta\geq \val\Dtheta$ always hold.
We say that \Qtheta has {\em zero duality gap} if $\val\Qtheta=
\val\Ptheta= \val\Dtheta$.
Throughout this paper we denote by~$\bar{\theta}$ a nominal value of the parameter~$\theta$, such that \Qthetabar  has zero duality gap.

\begin{definition}[SDP Stability]
  The family $\Qtheta$ is \emph{SDP-stable} near $\bar{\theta}$
  if there exists a neighborhood of $\bar{\theta}$ where $\val\Qtheta=\val\Ptheta= \val\Dtheta$.
\end{definition}

In this paper we will derive theorems that establish SDP-stability of \Qtheta near~$\bar\theta$,
so that the optimal value of the SDP relaxation agrees with that of the QCQP when $\theta$ is close to $\bar\theta$.
Moreover, the conditions in our theorems also guarantee that
the SDP relaxation \Ptheta recovers the minimizer of the QCQP.

\subsection*{Contributions and structure of the paper}

We start this paper, in \Cref{s:zerodualitygap}, by recalling a simple sufficient condition for zero duality gap in QCQPs.
As a corollary, the SDP relaxation of the nearest point problem to a quadratic variety is tight,
when the observed point is on the variety.

Our main contribution consists of two stability results
that guarantee zero duality gap for \Qtheta when $\theta$ close to~$\bar \theta$.
Our first result is \Cref{thm:firstresulthom} in \Cref{s:firstresult}
which focuses on QCQPs in which the constraints are independent of~$\theta$
and the cost function satisfies a certain strict convexity property.
\Cref{thm:firstresulthom} shows that a natural regularity condition (constraint qualification) on the optimal solution of \Qthetabar guarantees SDP-stability nearby~$\bar\theta$.
An important special instance is the nearest point problem
$\min\{ \|y \!-\! \theta \|^2 : y \in \mathbf{Y}\}$
to a quadratic variety~$\mathbf{Y}$.
\Cref{thm:firstresultnearest} shows SDP-stability nearby any point $\bar\theta \in \mathbf{Y}$ which is sufficiently regular.
Moreover, \Cref{thm:boundnearest} provides explicit bounds on the allowed perturbations.

Our second result is \Cref{thm:main} in \Cref{s:main},
which addresses the general family \Qtheta introduced above.
In particular, it allows perturbations in the constraints, and nonconvex cost functions.
\Cref{thm:main} shows that appropriate regularity conditions, and a technical restricted Slater condition that we introduce,
guarantee SDP-stability near~$\bar\theta$.
The restricted Slater condition is non-trivial (see \Cref{exmp:twistedcubicbad}). However, it can be checked efficiently by solving an SDP.

A large number of statistical estimation problems fall under the umbrella of our results,
and in each of these cases we can show that the SDP relaxation can solve the nonconvex QCQP in the {low-noise regime}.
\Cref{s:applications} illustrates the following applications:
\emph{(machine learning)} tensor principal component analysis,
\emph{(computer vision)}
triangulation problem, camera resectioning, essential matrix estimation,
\emph{(robotics)}
rotation synchronization, $SE(d)$ synchronization, orthogonal Procrustes,
\emph{(control theory)}
system identification,
\emph{(symbolic computation)} approximate GCD.
Collectively, \Cref{thm:firstresulthom,thm:main} explain many special instances in the literature
where the above mentioned zero duality gap phenomenon has been observed under low noise, fulfilling our stated goal.

Even though our theory is developed for QCQPs,
the results can be applied to analyze more general SDP relaxations of polynomial optimization problems.
In particular, they can be used to analyze SDP relaxations based on the \emph{sum-of-squares} (SOS) method.
This is illustrated in \Cref{s:sos}.
A highlight is \Cref{thm:sos}, that applies our theory to derive a stability result for the SOS method in unconstrained polynomial optimization.

\subsection*{Related work}
Our work straddles two lines of inquiry within  nonlinear optimization.

The first is the use of SDP relaxations to solve nonlinear, nonconvex optimization problems~\cite{Luo2010,Blekherman2013}.
There are several results concerning conditions under which SDP relaxations are tight \cite{Zhang2000,Ye2003,Beck2006,Kim2003}.
Classes of QCQPs such as trust-region problems~\cite{Stern1995}, S-lemma type problems~\cite{Polik2007}, and combinatorial optimization problems~\cite{Gouveia2010} have been well investigated, and
these references are far from exhaustive.

The second is the  perturbation theory of nonlinear programs. This is also a mature area of work~\cite{Bonnans2013,Fiacco1990,Levy2000}.
In particular, sufficient conditions for continuity and differentiability of the optimal value/solutions are known~\cite{Bonnans2013,Fiacco1990}.
Similarly, the Lipschizian stability of general optimization problems, together with concepts such as tilt/full stability, has received a lot of attention~\cite{Levy2000,Mordukhovich2013}.

The subject of this paper is the tightness of SDP relaxations as the QCQP is perturbed. This topic has attracted attention over the past decade as SDPs have been used successfully to solve estimation problems in practice, but a theoretical explanation of their efficacy has remained elusive. However, a number of problem specific tightness results have appeared. In particular, it was shown that the SDP relaxations of the triangulation problem~\cite{Aholt2012} and the rotation synchronization problem are tight under low noise assumptions~\cite{Eriksson2018,Fredriksson2012}.
Our results provide a uniform framework to analyze these and a much broader class of QCQPs.

Since a sufficient condition for zero duality gap is that the SDP solution has rank-one, we may equivalently study whether the rank of the minimizer is preserved after small perturbations of~$\theta$.
Past work (e.g. \cite{Freund2004,Nayakkankuppam1999}) established this rank stability assuming a unique primal/dual optima and strict-complementarity.
However, for SDPs coming from polynomial optimization problems the dual solution may not be unique (due to redundant constraints) and strict complementarity may not hold, so these results cannot be applied.
We are not aware of previous results that avoid these hypotheses.

\section{Zero Duality Gap For QCQPs}\label{s:zerodualitygap}

We begin by recalling a sufficient condition for zero duality gap
in QCQPs that will be useful in later sections.
Consider the homogeneous QCQP:
\begin{equation*} \label{eq:Q}
\tag{\textit{Q}}
\begin{aligned}
  \min_{x\in\RR^N}\quad
  & f(x) :=x^\top F x \\
  \text{s.t.}\quad
  & g^i(x):= x^\top G^i x + c^i =0,\quad i=1,\dots,m
\end{aligned}
\end{equation*}
where $F, G^i\in \SR^N$ are symmetric $N\times N$ matrices,
and $c^i\in \RR$ are scalars, with at least one of them nonzero.
Recall that the dual pair of SDP relaxations of \eqref{eq:Q} are:

\begin{minipage}{.50\linewidth}
\begin{equation*}\label{eq:P}
\tag{\textit{{P}}}
\begin{aligned}
  \min_{S\in\SR^N}\quad
  & {F}\bullet S \\
  \text{s.t.}\quad
  & G^i \bullet S + c^i = 0, \quad i=1,\dots,m\\
  & S \succeq 0
\end{aligned}
\end{equation*}
\end{minipage}\qquad
\begin{minipage}{.40\linewidth}
\begin{equation*}\label{eq:D}
\tag{\textit{{D}
}}
\begin{aligned}
  \max_{\lambda\in \RR^m}\quad
  &\textstyle\sum_i \lambda_i c^i \\
  \text{s.t.}\quad
  & \mathcal{H}(\lambda)
  \succeq 0
\end{aligned}
\end{equation*}
\end{minipage}

\noindent
where $\mathcal{H}(\lambda) \in \SR^N$ is half the Hessian of the Lagrangian function $L(x,\lambda)$:
\begin{align*}
  L(x, \lambda) :=
  f + \sum_i \lambda_i g^i,
  \qquad
  \mathcal{H}(\lambda) := \tfrac{1}{2}\,\nabla^2_{xx} L(x, \lambda) =
  F + \sum_i \lambda_i G^i.
\end{align*}

Given a feasible solution $x$ of~\eqref{eq:Q}, $\lambda\in\RR^m$ is a \emph{Lagrange multiplier} at~$x$ if
\begin{align} \label{eq:Lagrange multiplier}
  \nabla_x L(x, \lambda) = 0,
  \qquad\text{or equivalently}\qquad
  \mathcal{H}(\lambda)x
  = F x + \sum\nolimits_i \lambda_i G^i x
  = 0.
 \end{align}
We denote by $\Lambda(x)$ the affine space in $\RR^m$ of Lagrange multipliers at~$x$.
We say that $x$ is a \emph{critical point} of \eqref{eq:Q} if it is feasible and $\Lambda(x)$ is nonempty.
It is known that all local minima of~\eqref{eq:Q} are critical points under appropriate regularity conditions
(see e.g.,~\cite[\S5]{Bazaraa2013}).

The following lemma gives a sufficient condition for zero duality gap for  \eqref{eq:Q}.
\begin{lemma}\label{lem:zerodualitygap}
  Consider the QCQP~\eqref{eq:Q}.
  Let $x\in \RR^N, \lambda\in \RR^m$ be such that:
  \begin{enumerate}[label=(\roman*)]
    \item $g^i(x) = 0$ for $i=1,\ldots,m$ (
    $x$ is primal feasible).
    \item $\mathcal{H}(\lambda)\succeq 0$ ($\lambda$ is dual feasible).
    \item $\mathcal{H}(\lambda) x = 0$ ($\lambda$ is a Lagrange multiplier at $x$).
  \end{enumerate}
  Then $x$ is optimal for~\eqref{eq:Q}, $\lambda$ is optimal for~\eqref{eq:D}, and $\val\eqref{eq:Q}=\val\eqref{eq:D}$.
If in addition, $\mathcal{H}(\lambda)$ has corank one,
   then $x x^\top$ is the unique optimum of~\eqref{eq:P}, and $x$ is the unique optimum
   of~\eqref{eq:Q}, up to sign.
\end{lemma}

\begin{proof}
  Since $\mathcal{H}(\lambda) x = 0$ and $x$ is primal feasible then
  \begin{align*}
    0 \,=\, x^\top \mathcal{H}(\lambda) x
    \,=\, x^\top F x + \sum\nolimits_i \lambda_i\, x^T G^i x
    \,=\, x^\top F x - \sum\nolimits_i \lambda_i c^i,
  \end{align*}
  so the primal value of $x$ equals the dual value of $\lambda$.
  By weak duality we have that $x$ and $\lambda$ are primal/dual optimal.

  Suppose $S$ is an optimal solution of~\eqref{eq:P}. Then $S \neq 0$ since at least one $c^i \neq 0$.
  By complementary slackness, $\mathcal{H}(\lambda)\bullet S = 0$, and since $\mathcal{H}(\lambda)$ and $S$ are both positive semidefinite,
  $\rank\mathcal{H}(\lambda) \!+\! \rank S\leq N$.
  If $\corank \mathcal{H}(\lambda) \!=\! 1$ then $\rank S \!=\! 1$, so any optimal solution of \eqref{eq:P} has rank one.
  This implies that
  \eqref{eq:P} has a unique optimal solution
  $S$ since if there was a second optimal solution $S'$,
  then there would be an element in their convex hull of
  rank two that is also optimal for \eqref{eq:P}.
  Since every rank one optimal solution
  of \eqref{eq:P} corresponds to an optimal solution of
  \eqref{eq:Q}, it must be that $x$ is the
  unique optimal solution of \eqref{eq:Q} up to sign,  and $S = x x^\top$.
\end{proof}

\Cref{lem:zerodualitygap} is well known, see e.g., \cite[Thm~2]{Aholt2012}.
  A generalization to inequality constrained QCQPs is given in \cite{Zheng2012}.
We have given a complete proof here since this lemma is an important tool
in this paper.

The main results in this paper are all stated for homogeneous QCQPs, while in many applications the QCQPs that arise are inhomogeneous. We take a moment to discuss the effect of homogenization.

One can always rid a quadratic equation
of linear terms by introducing a {homogenizing} variable~$z_0$ where
$z_0^2=1$;
the inhomogeneous quadratic polynomial
$\tilde{g}(y) = y^\top \tilde{G} y \!+\! 2\tilde{l}^\top y \!+\! \tilde{c}$
homogenizes to
\begin{align*}
  y^\top \tilde{G} y +  2\tilde{l}^\top y z_0  + \tilde{c} z_0^2
  = (z_0, y^\top) \begin{pmatrix} \tilde{c} & \tilde{l}^\top\\
    \tilde{l} & \tilde{G} \end{pmatrix} \begin{pmatrix} z_0 \\
  y \end{pmatrix} = x^\top G x = g(x),
  \qquad
  z_0^2=1,
\end{align*}
where $x=(z_0,y)$.
For instance, the homogenized form of~\eqref{eq:twistedcubic} is:
  \begin{equation*}
  \begin{aligned}
    \min_{z_0\in \RR,\,y\in\RR^3}\quad & \|y-\theta z_0\|^2,
    \qquad \text{ s.t. }\qquad
    z_0^2 = 1,\quad
    y_2z_0 = y_1^2,\quad
    y_3z_0 = y_1 y_2
  \end{aligned}
  \end{equation*}
which is a problem in the variables $x = (z_0,y_1,y_2,y_3)$.
There is a 1:2 correspondence between the solutions of an inhomogeneous
QCQP and its homogenization:
$y$~is a solution of an inhomogeneous
QCQP if and only if $\pm(1,y)$ are solutions of its homogenization. Therefore we can assume that $z_0 = 1$. The optimal objective function values of both problems coincide and
hence there is no loss of generality in assuming the homogeneous form.

We now apply \Cref{lem:zerodualitygap} to the following inhomogeneous QCQP:
\begin{align*} \label{eq:nearestpoint}
   \tag{$E_\theta$}
   \min_{y\in\RR^n} \quad \| y - \theta\|^2,
   \qquad\text{s.t.}\qquad
   \tilde{g}^i(y) \!=\! 0,\quad i \!=\! 1,\ldots, m.
\end{align*}
This is the nearest point problem for an algebraic variety
$\mathbf{Y}$ cut out by quadratic polynomials, i.e.,
$$ \mathbf{Y} = \{ y \in \RR^n \,:\, \tilde{g}^i(y) \!=\! 0,\;\; i \!=\! 1,\ldots, m \}.$$
We will prove that for a parameter $\bar\theta$ on the variety $\mathbf{Y}$,
the problem \Ethetabar has zero duality gap.
The optimal solution of \Ethetabar is obviously~$\bar\theta$,
but the point is that this can be recognized by its SDP relaxation,
which  will be useful when we consider families of nearest point problems in the next section.

\begin{corollary}[Nearest point to a quadratic  variety]
  \label{cor:nearestpoint-zerodualitygap-quadratic}
  The nearest point problem \Ethetabar, for the quadratic variety~$\mathbf{Y}$,
  has zero duality gap when $\bar\theta$ lies on~$\mathbf{Y}$.
\end{corollary}

\begin{proof}
  After homogenization, the problem \Ethetabar becomes:
  \begin{equation*}
    \begin{aligned}
      \min_{x\in\RR^{n+1}}\quad
      & f_{\bar\theta}(x) = \|y -\bar\theta z_0\|^2,
      \qquad \text{s.t.}\qquad
      z_0^2 \!=\! 1,\qquad
      g^i(x) \!=\! 0,\quad i \!=\! 1, \ldots, m,
    \end{aligned}
  \end{equation*}
  where $x = (z_0,y)$ and $g^i$ is the homogenization of $\tilde{g}^i$.
  An optimal solution of the homogeneous QCQP is
  $\bar{x}=(1,{\bar\theta})$ since $\bar\theta$ lies on $\mathbf{Y}$.
  Check that this $\bar{x}$ and $\bar\lambda \!=\! 0$ satisfy the conditions of \Cref{lem:zerodualitygap}.
  The proof relies on the fact that
  $F_{\bar\theta} =
  \Bigl(\begin{smallmatrix} \|\bar\theta\|^2 & - {\bar\theta}^\top \\ - \bar\theta & I \end{smallmatrix}\Bigr)
  \succeq 0,$
  and that since the optimal value is $\bar{x}^\top F_{\bar\theta} \bar{x} = 0$,
  we have that $F_{\bar\theta} \bar{x} = 0$.
\end{proof}

\section{A Special Case}\label{s:firstresult}
In this section we establish stability results for a simplified
version of the parametrized family \Qtheta from the Introduction.
In particular, we assume that the parameter $\theta$ only appears in the objective function and not in the constraints.
More precisely, we consider the family of QCQPs:
\begin{align}\label{eq:Qthetaobj}
\tag{$Q_\theta^{\mathrm{obj}}$}
  \min_{x\in\RR^N}\quad &
  f_\theta(x) := x^\top F_{\theta} x,
  \qquad\text{s.t.}\qquad
  g^i(x)\!:=\! x^\top G^i x + c^i \!=\!0, \quad i \!=\! 1,\ldots,m.
\end{align}
We also assume throughout this section that the nominal parameter~$\bar\theta$ is such that $F_{\bar\theta} \succeq 0$,
$\corank F_{\bar\theta} \!=\! 1$,
and the optimal value $\val(Q_{\bar\theta}^{\mathrm{obj}}) = 0$. These assumptions imply that \Qthetabarobj has zero duality gap and that it has a unique optimal solution $\bar{x}$ that can be
recovered by its SDP relaxation $(P_{\bar \theta}^{\mathrm{obj}})$.
As in the proof of \Cref{cor:nearestpoint-zerodualitygap-quadratic},
check that any optimal solution $\bar{x}$ of \Qthetabarobj
and $\bar\lambda \!=\! 0$ satisfy the conditions of \Cref{lem:zerodualitygap}.

We will prove that~\Qthetaobj is SDP-stable near~$\bar\theta$.
As a corollary we will obtain a stability result for the nearest point problem on a quadratic variety.
We will also derive bounds on the magnitude of the perturbations that can be tolerated.

In order to prove SDP-stability,
we require a regularity assumption on the optimal solution $\bar{x}$ of~\Qthetabarobj.
Several regularity conditions (constraint qualifications) have been studied in optimization.
We will rely on the following, which is one of the weakest.

\begin{definition}\label{defn:ACQ}
  Given $g : \RR^N\!\to\! \RR^m$, let $\mathbf{X}:= \{x\!\in\! \RR^N: g(x) \!=\! 0 \}$.
  The \emph{Abadie constraint qualification} (ACQ) holds at~$x\in \mathbf{X}$, denoted $\ACQ{\mathbf{X}}{x}$,
  if $\mathbf{X}$ is a smooth manifold nearby $x$,
  and $\rank\nabla{g}(x) = \codim_x \mathbf{X}$.
  Here $\codim_x \mathbf{X}:= N\!-\!\dim_{x} \mathbf{X}$ denotes the local codimension of~$\mathbf{X}$ at~$x$,
  and $\nabla g$ denotes the Jacobian matrix.
\end{definition}

It is well-known that $\ACQ{\mathbf{X}}{x}$ guarantees the existence of Lagrange multipliers at~$x$ (see e.g.,~\cite[\S5.1]{Bazaraa2013}).
In this paper we are only interested in the case where $g$ is a polynomial map ($\mathbf{X}$ is an algebraic variety).
In that case the condition $\rank\nabla{g}(x)= \codim_{x} \mathbf{X}$
implies that $x$ is a \emph{smooth} point of~$\mathbf{X}$
(see e.g., \cite[\S3.3]{Bochnak2013}),
and consequently, $\ACQ{\mathbf{X}}{x}$ holds if and only if
$\rank\nabla{g}(x)= \codim_{x} \mathbf{X}$.

The main result of this section is the following stability
theorem for homogeneous QCQPs:
\begin{theorem}\label{thm:firstresulthom}
  Consider the family \Qthetaobj,
  where $F_\theta$ is a continuous function of~$\theta$ and $c\!\neq\! 0$.
  Let $\bar{\theta}$ be such that
  $F_{\bar{\theta}} \succeq 0$, has corank one,
  and $\val(Q_{\bar\theta}^{\mathrm{obj}}) \!=\! 0$.
  If further, $\ACQ{\mathbf{X}}{\bar{x}}$ holds,
  where
  $\mathbf{X} = \{ x \!\in\! \RR^N : g(x) \!=\! 0\}$
  is the feasible set,
  then $(Q_{\theta}^{\mathrm{obj}})$ is SDP-stable near~$\bar\theta$,
  and its primal SDP relaxation $(P_\theta^{\mathrm{obj}})$ recovers its minimizer.
\end{theorem}

\begin{remark}
  Observe that $F_{\bar\theta} \!\succeq\! 0$, $\val(Q_{\bar\theta}) \!=\! 0$ if and only if $\bar\lambda \!=\! 0$ is optimal for~\Dthetabar.
  This assumption is nonrestrictive,
  as we can always ensure that $\bar\lambda \!=\! 0$ by adding to the cost function a linear combination of the constraints.
  The crucial assumptions of \Cref{thm:firstresulthom} are that $\mathcal{H}_{\bar\theta}(\bar\lambda) \!=\! F_{\bar\theta}$ has corank one and $\ACQ{\mathbf{X}}{\bar x}$ holds.
\end{remark}

In order to prove the theorem,
we first derive a tool for establishing SDP-stability near $\bar \theta$,
that applies not just to the QCQPs in this section,
but to the general family \Qtheta from the Introduction.

Consider the \emph{Lagrange multiplier mapping}:
\begin{equation}\label{eq:lagrangeset}
  \begin{aligned}
    \mathfrak{L}: \Theta \rightrightarrows \RR^N\times \RR^m,
    \qquad
    \theta\mapsto\,
    &\{(x_\theta,\lambda_\theta): x_\theta \text{ feasible for \Qtheta},\; \lambda_\theta\in \Lambda_\theta(x_\theta)\} \\
    &=\{(x_\theta,\lambda_\theta): g_\theta(x_\theta) {=} 0,\; \mathcal{H}_\theta(\lambda_\theta)x_\theta {=} 0\}.
  \end{aligned}
\end{equation}
As we will see, continuity properties of $\mathfrak{L}$ play a crucial role in
stability.

\begin{definition}
  The Lagrange multiplier mapping $\mathfrak{L}$ is \emph{weakly continuous}
  at a pair $\bar\ell = (\bar x, \bar\lambda) \in \mathfrak{L}(\bar\theta)$
  if there exists $\ell_\theta \in \mathfrak{L}(\theta)$ such that $\ell_\theta \to \bar\ell$ as $\theta\to\bar\theta$.
\end{definition}

\begin{proposition}
  \label{prop:stability}
  Let $\bar\theta$ be a zero duality gap parameter,
  and let $(\bar{x},\bar\lambda)$ primal/dual optimal of \Qthetabar.
  Suppose that $\mathcal{H}_{\bar\theta}(\bar\lambda)$ has corank one
  and that the mapping $\mathfrak{L}$ is weakly continuous at~$(\bar x,\bar\lambda)$.
  Then \Qtheta is SDP-stable near~$\bar\theta$
  and \Ptheta recovers its minimizer.
\end{proposition}
\begin{proof}
  By weak continuity, there exists $(x_\theta,\lambda_\theta)$ with  $x_\theta$ feasible for \Qtheta,
  $\lambda_\theta \!\in\! \Lambda_\theta(x_\theta)$,
  such that
  $(x_\theta,\lambda_\theta) \to (\bar{x},\bar\lambda)$ as $\theta\to \bar\theta$.
  It follows that $\mathcal{H}_{\theta}(\lambda_\theta) \to \mathcal{H}_{\bar\theta}(\bar\lambda)$,
  since $f_\theta$ and $g^i_\theta$ depend continuously on~$\theta$.
  Observe that $\mathcal{H}_{\theta}(\lambda_\theta)$ has a 0-eigenvalue since the Lagrange multiplier relationship implies $\mathcal{H}_{\theta}(\lambda_\theta)x_\theta = 0$.
  Also note that $\mathcal{H}_{\bar\theta}(\bar\lambda) \succeq 0$
  as $\bar\lambda$ is dual feasible.
  By assumption $\corank\mathcal{H}_{\bar\theta}(\bar\lambda) \!=\! 1$,
  and hence $\mathcal{H}_{\bar\theta}(\bar\lambda)$ has $N{-}1$ positive  eigenvalues.
  Since $\mathcal{H}_{\theta}(\lambda_\theta) \to \mathcal{H}_{\bar\theta}(\bar\lambda)$, by continuity of eigenvalues, $\mathcal{H}_{\theta}(\lambda_\theta)$ also has $N{-}1$ positive eigenvalues when $\theta\to \bar\theta$.
  Therefore,  $\mathcal{H}_{\theta}(\lambda_\theta) \succeq 0$ and
  $\corank\mathcal{H}_{\theta}(\lambda_\theta)=1$.
  This concludes the proof by \Cref{lem:zerodualitygap}.
\end{proof}

We will now prove that weak continuity holds for~\Qthetaobj at $(\bar{x}, \bar{\lambda})$, with $\bar\lambda=0$.
We first show that $x_\theta^* \rightarrow \bar{x}$ as $\theta\to\bar\theta$,
where $x_\theta^*$ is an optimal solution of \Qthetaobj and $\bar{x}$ is the unique optimal solution of \Qthetabarobj.
We make use of the following well-known lemma, see e.g., \cite[Prop.4.4]{Bonnans2013}.

\begin{lemma}\label{thm:continuity}
  Let $F: S \times \Theta \to \RR$ be a continuous function,
  where $S \subset \RR^N$ is a compact set.
  Then the function $f:\Theta\to \RR$ such that $\theta \mapsto \min_{x \in S} F(x,\theta)$ is continuous.
\end{lemma}

We denote the operator norm of a matrix~$A$ by $\|A\|$ and the Frobenius norm by $\|A\|_F$.

\begin{lemma}\label{thm:primalconvergence}
  For each~$\theta$, let $x_\theta^*$ be an optimal solution of \Qthetaobj.
  Then $x_\theta^*$ converges to $\bar{x}$, up to sign.
\end{lemma}

\begin{proof}
  Set $x\!=\! (x_1,y)$, where $y \!=\! (x_2,\dots,x_N) \!\in\! \RR^{N-1}$. Since at least one $c^i \neq 0$, we may assume
  that $c^1 \!=\! -1$.
  Since $F_{\bar \theta}$ is positive semidefinite and has corank one, we may also
  assume after a change of coordinates that
  $f_{\bar\theta}(x) \!=\! \|y\|^2$ and
  $\bar{x} \!=\! (1,0)$.

  We first show that any point $x= (x_1,y)$,
  feasible for~\Qthetaobj,
  satisfies $\|x\| \!\leq\! \alpha (\|y\| {+} 1)$
  for some constant $\alpha \!\geq\! 1$ that only depends on~$g^1$.
  As $g^1(\bar{x}) \!=\! \bar{x}^\top G^1 \bar{x} {-} 1 \!=\! 0$,
  the top left entry of $G^1$ is a one.
  Hence, we may rewrite the first equation in the form
  $g^1(x) = x_1^2 {-} 2 (v^\top y)x_1 \!-\! (y^\top \tilde{V} y {+} 1)$
  for some $v \!\in\! \RR^{N-1}$, $\tilde{V} \!\in\! \SR^{N-1}$,
  or equivalently,
  $g^1(x) = (x_1 {-} v^\top y)^2 \!-\! (y^\top V y {+} 1)$
  for $V \!:=\! \tilde{V} {+} v v^\top$.
  The equation $g^1(x)\!=\!0$ implies
  $x_1 = v^\top y \!\pm\! \sqrt{y^\top V y {+} 1}$.
  It follows that $|x_1| \!\leq\!
  (\|v\| {+} \|V\|^{1/2})\|y\|\!+\!1$,
  and hence $\|x\| \!\leq\! \alpha (\|y\| {+} 1)$ for
  $\alpha := 1{+}\|v\| {+} \|V\|^{1/2}$.

  From now on we assume that $\theta$ is sufficiently close to $\bar\theta$
  so that $\|F_\theta \!-\! F_{\bar\theta}\| < 1/(8\alpha^2)$.
  We claim that any optimal solution $x_\theta^*$ of
  \Qthetaobj belongs to the compact set
  $S := \{ x{=}(x_1,y) : \|y\| {\leq} 1, \|x\| {\leq} 2\alpha \}$.
  Indeed, any feasible point $x = (x_1,y)$, with $\|y\| {\geq} 1$, has a large cost value:
  \begin{align*}
    f_\theta(x)
    &\geq\,
    f_{\bar\theta}(x) - |f_{\theta}(x)\!-\!f_{\bar\theta}(x)|
    \,\geq\,
    \|y\|^2 - \|F_{\theta} \!-\! F_{\bar\theta}\| \cdot \|x\|^2
    \\
    &\geq\,
    \|y\|^2 - \tfrac{1}{8 \alpha^2} \cdot \alpha^2(\|y\|{+}1)^2
    \,\geq\, (1  - \tfrac{1}{8} \cdot 4)\|y\|^2
    \,\geq\, 1/2.
  \end{align*}
  Such a point $x$ cannot be optimal for \Qthetaobj because $\bar{x}$ has a lower cost:
  \begin{align*}
    f_\theta(\bar{x})
    &\leq\,
    f_{\bar\theta}(\bar{x}) + |f_{\theta}(\bar{x})\!-\!f_{\bar\theta}(\bar{x})|
    \,\leq\,
    \|F_{\theta} \!-\! F_{\bar\theta}\| \cdot \|\bar{x} \|^2
    \,\leq\,
    1/(8 \alpha^2).
  \end{align*}

  As the optimal solutions of \Qthetaobj belong to the compact set~$S$ when
  $\theta$ is sufficiently close to $\bar \theta$,
  we may apply \Cref{thm:continuity}.
  We conclude that
  $f_\theta(x_\theta^*) \to f_{\bar\theta}(\bar{x}) \!=\! 0$
  as $\theta\to\bar\theta$.
  Denoting by $\|\cdot\|_S$ the infinity norm on~$S$, then
  \begin{align*}
    \|y_\theta^*\|^2
    \,=\,
    f_{\bar\theta}(x_\theta^*)
    \,\leq\,
    |f_{\bar\theta}(x_\theta^*) \!-\! f_{\theta}(x_\theta^*) |
    + |f_\theta(x_\theta^*) |
    \,\leq\,
    \|f_{\bar\theta} - f_\theta\|_S
    + |f_\theta(x_\theta^*) |
    \xrightarrow{\theta\to\bar\theta}
    0.
  \end{align*}
  Therefore,
  $y_\theta^* \to 0$ as $\theta \to \bar\theta$.
  Recall that a feasible point $x{=}(x_1,y)$ satisfies
  $x_1 = v^\top y \!\pm\! \sqrt{y^\top V y {+} 1}$.
  It follows that $x_\theta^*$ converges to $\bar{x} = (1,0)$, up to sign.
\end{proof}

We now show that $\lambda_\theta \rightarrow 0$ as $\theta \rightarrow \bar \theta$. For this we rely on the ACQ property of the quadratic variety $\mathbf{X}$ at~$\bar x$.
Let $\nabla{g}$ denote the Jacobian of $g = (g^1, \ldots, g^m)$.

\begin{lemma}\label{thm:dualconvergence}
  Let $x_\theta$ be a critical point of \Qthetaobj.
  Let $\sigma_\theta$ be the $s$-th largest singular value of the Jacobian $\nabla{g}(x_\theta)$, where $s:=\codim_{x_\theta}\! \mathbf{X}$.
  \begin{enumerate}[label=(\roman*)]
    \item
      If $\ACQ{\mathbf{X}}{x_\theta}$ holds, then
      there exists $\lambda_\theta \!\in\! \Lambda_\theta(x_\theta)$ with
      $\|\lambda_\theta\| \leq \frac{1}{\sigma_\theta}\|\nabla f_{\theta}(x_\theta)\|$.
    \item
      If $\ACQ{\mathbf{X}}{\bar{x}}$ holds and $x_\theta \to \bar x$ as $\theta \to \bar\theta$,
      then there exists $\lambda_\theta \!\in\! \Lambda_\theta(x_\theta)$ such that $\lambda_\theta \to 0$.
  \end{enumerate}
\end{lemma}

\begin{proof}
  \emph{(i)}
  Let $J_\theta:= \nabla g(x_\theta)$ be the Jacobian at $x_\theta$.
  If $\ACQ{\mathbf{X}}{x_\theta}$ then $\rank J_\theta \!=\! s$,
  and hence $\sigma_\theta \!>\! 0$.
  Recall that $\Lambda_\theta(x_\theta)$ is the solution space of
  the linear system
  $\lambda^\top J_\theta = -\nabla f_{\theta} (x_\theta)$.
  The linear system has a solution as $x_\theta$ is a critical point.
  Hence,
  $\lambda_{\theta}^\top:= -\nabla f_{\theta}(x_\theta) J_\theta^\dagger$
  is one such solution, where $J_\theta^\dagger$ denotes the pseudo-inverse of~$J_\theta$. The first part of the lemma follows by noticing that $\|J_\theta^\dagger\|=\nobreak 1/{\sigma_\theta}$.

  \emph{(ii)}
  Since ACQ is an open condition, if it holds at $\bar{x}$, then
  it also holds in a neighborhood of $\bar{x}$ and hence at $x_\theta$, for $\theta$ sufficiently close to $\bar \theta$.
  By assumption, $\nabla f_{\bar\theta}(\bar{x})= 2 F_{\bar\theta} \bar x = 0$,
  and by ACQ, $\sigma_{\bar\theta} >0$.
  It follows that $\|\lambda_\theta\| \leq \frac{1}{\sigma_\theta} \|\nabla f_\theta(x_\theta)\| \to 0$ as $\theta \to \bar\theta$,
  and hence $\lambda_\theta\to 0$.
\end{proof}

We are now ready to prove~\Cref{thm:firstresulthom}.
\begin{proof}[Proof of \Cref{thm:firstresulthom}]
We have shown that $\bar\lambda \!=\! 0$ is an optimal solution for
$(D_{\bar\theta})$,
and we are given that $\mathcal{H}_{\bar\theta}(0) \!=\! F_{\bar\theta}$ has corank one.
\Cref{thm:primalconvergence,thm:dualconvergence} show the existence of
$(x_\theta,\lambda_\theta) \in \mathfrak{L}(\theta)$ such that
$x_\theta \rightarrow \bar{x}$ and
$\lambda_\theta \rightarrow 0$ as $\theta \rightarrow \bar\theta$,
where $\bar{x}$ is the unique optimal solution of~\Qthetabarobj.
Then weak continuity also holds,
and the theorem follows from \Cref{prop:stability}.
\end{proof}

\subsection*{Specializations}
Having completed the proof of \Cref{thm:firstresulthom}, we now derive two specializations of it.
The first is the inhomogeneous version which is typically what one sees in applications.

\begin{theorem}\label{thm:firstresultinhom}
  Consider the problem
  \begin{align}\label{eq:Qthetaobjinhom}
    \tag{$\tilde{Q}_\theta^{\mathrm{obj}}$}
    \min_{y}\;\;
    &\tilde{f}_\theta(y)
    \!:=\! y^\top \tilde{F}_{\theta} y
    \!+\!  \tilde{l}_\theta^{\,\top} y \!+\! \tilde{c}_\theta,
    \quad\text{s.t.}\quad
    \tilde{g}^i(y)
    \!:=\! y^\top \tilde{G}^i y \!+\!
    (\tilde{l}^i)^{\!\top} y \!+\! \tilde{c}^i \!=\!0,
    \;\; i \!=\! 1,\ldots,m
  \end{align}
  where $\tilde{f}_\theta,\tilde{g}^i$ are quadratics,
  and $\tilde{f}_\theta$ depends continuously on $\theta$.
  Let $\bar\theta$ be such that $\tilde{f}_{\bar\theta}$ is strictly convex, and its unique unconstrained minimizer
  $\bar{y}$ is the minimizer of $(\tilde{Q}_{\bar\theta}^{\mathrm{obj}})$.
  If $\ACQ{\mathbf{Y}}{\bar{y}}$ holds,
  where $\mathbf{Y}$ is the feasible set,
  then $(\tilde{Q}_{\theta}^{\mathrm{obj}})$ is SDP-stable near~$\bar\theta$
  and $(\tilde{P}_{\theta}^{\mathrm{obj}})$ recovers its minimizer.
\end{theorem}

\begin{proof}
  The homogenized QCQP is:
  \begin{align*}
    \min_{x=(z_0,y)\in \RR^{n+1}}\;\; &f_\theta(x),
    \quad\text{ s.t. }\quad
    z_0^2\!=\!1, \quad
    g^i(x) \!=\!0, \;\; i \!=\! 1,\ldots,m
  \end{align*}
  where $f_\theta, g^i$ are the homogenizations of
  $\tilde{f}_\theta$, $\tilde{g}^i$.
  We need to show that the conditions of~\Cref{thm:firstresulthom} are satisfied.
  We may assume that $f_{\bar\theta}(\bar{x}) = \tilde{f}_{\bar\theta}(\bar{y}) = 0$
  after possibly shifting the cost function.
  Since $\tilde{f}_{\bar\theta}$ is strictly convex,
  then $f_{\bar\theta}$ is convex and its Hessian has corank one.
  It remains to see that $\ACQ{\mathbf{Y}}{\bar{y}}$ implies $\ACQ{\mathbf{X}}{\bar{x}}$.
  This follows from the equations:
  $\rank\nabla g(\bar{x}) = \rank\nabla \tilde{g}(\bar{y}) \!+\!1$,
  $\dim_{\bar{x}}\mathbf{X} = \dim_{\bar{y}}\mathbf{Y}$,
  and $\codim_{\bar{x}}\mathbf{X} = \codim_{\bar{y}}\mathbf{Y}\!+\!1$,
  which are easy to verify.
\end{proof}

As mentioned in the Introduction,
several estimation problems reduce to minimizing the Euclidean distance from a point $\theta$ to a quadratic variety.
As a corollary of \Cref{thm:firstresultinhom}, we get that
such nearest point problems have a tight SDP relaxation when $\theta$ is close enough to the variety.

\begin{corollary}\label{thm:firstresultnearest}
  Consider the problem~\Etheta.
  Let $\bar{\theta} \!\in\! \mathbf{Y}$ be such that $\ACQ{\mathbf{Y}}{\bar\theta}$ holds.
  Then \eqref{eq:nearestpoint} is SDP-stable near~$\bar\theta$ and the SDP recovers the minimizer.
\end{corollary}

The above corollary corresponds to the special
case of \Cref{thm:firstresultinhom} in which the
objective function is $\tilde{f}_\theta(y) := \|y\!-\!\theta\|^2$.
Indeed, this objective is strictly convex,
the minimizer is $\bar{y} \!=\! \bar\theta$ (since $\bar\theta \!\in\! \mathbf{Y}$),
which is also the unconstrained minimizer.
\Cref{thm:firstresultnearest} generalizes the main result of~\cite{Aholt2012},
as will be discussed in \Cref{exmp:triangulation}.

\subsection*{Guaranteed region of SDP-stability}
The \emph{SDP-exact region} of the family of QCQPs~\Qthetaobj
is the set of all parameters~$\theta$ for which the SDP relaxation solves the problem exactly,
i.e., such that $(Q_\theta^{\mathrm{obj}})$ has zero duality gap
and $(P_\theta^{\mathrm{obj}})$ recovers the minimizer.
The SDP-exact region is a semialgebraic set,
provided that the dependence on~$\theta$ is algebraic.
It can be computed exactly using tools from computer algebra~\cite{Cifuentes2019}, though the computation is quite expensive.
Our goal is to find an explicit neighborhood of $\bar\theta$ that is entirely contained in this region.

The next theorem gives a simple criterion to guarantee that a parameter $\theta$ belongs to the SDP-exact region.

\begin{theorem}\label{thm:boundhom}
  Consider the setting of \Cref{thm:firstresulthom}.
  Let $\bar\theta$ be a zero duality gap parameter, and let $\bar x$ be the minimizer.
  Let $\theta$ be another parameter for which $\bar{x}$ is a critical point of~\Qthetaobj,
  and also
  \begin{align*}
      \frac{1}{\sigma_s}\,
      \| \mathcal{G}\|\, \|\nabla f_{\theta}(\bar{x})\|
      + \|F_\theta {-} F_{\bar\theta}\|
      \,<\, \nu_{2}(F_{\bar\theta})
  \end{align*}
  where $s \!=\! \codim_{\bar{x}}\mathbf{X}$,
  $\sigma_s \!=\! \sigma_s(\nabla g(\bar x))$ is the $s$-th smallest singular value of the Jacobian,
  $\nu_2(\,)$ denotes the 2nd smallest eigenvalue,
  and $\| \mathcal{G}\|$ is the operator norm of the linear map $\mathcal{G}(\lambda) \!:=\! \sum_{i=1}^m \lambda_i G^i$.
  Then $(Q_\theta^{\mathrm{obj}})$ has zero duality gap
  and $(P_\theta^{\mathrm{obj}})$ recovers the minimizer.
\end{theorem}

\begin{proof}
  By \Cref{thm:dualconvergence},
  there exists a Lagrange multiplier $\lambda_\theta$ for~\Qthetaobj such that
  $\|\lambda_\theta\| \leq \frac{1}{\sigma_s} \|\nabla{f}_\theta(\bar{x})\| $.
  We will show that $\mathcal{H}_\theta(\lambda_\theta) \succeq 0$,
  and hence \Qthetaobj has zero duality gap by \Cref{lem:zerodualitygap}.
  It suffices to prove that $\nu_2(\mathcal{H}_\theta(\lambda_\theta))>0$,
  as the first eigenvalue is zero.
  Recall that $\lambda \!=\! 0$ is a Lagrange multiplier for~\Qthetabarobj.
  By Weyl's inequality, we have
  \begin{align*}
    \nu_{2}(\mathcal{H}_{\bar\theta}(0)) \!-\! \nu_{2}(\mathcal{H}_\theta(\lambda_\theta))
    \,&\leq\,
    \| \mathcal{H}_{\bar\theta}(0) \!-\! \mathcal{H}_\theta(\lambda_\theta)\|
    \,\leq\,
    \| \mathcal{H}_{\bar\theta}(0) \!-\! \mathcal{H}_\theta(0)\|
    +
    \| \mathcal{H}_\theta(0) \!-\! \mathcal{H}_\theta(\lambda_\theta) \|_F
    \\
    &=\,
    \|F_{\bar\theta} \!-\! F_{\theta}\|
    +
    \|\mathcal{G}(\lambda_\theta)\|_F
    \,\leq\,
    \|F_{\bar\theta} \!-\! F_\theta\|
    +
    \frac{1}{\sigma_s} \|\mathcal{G}\| \, \|\nabla{f}_\theta(\bar{x})\|
  \end{align*}
  where we used that
  $\|\lambda_\theta\| \leq \frac{1}{\sigma_s} \|\nabla{f}_\theta(\bar{x})\| $
  in the last equation.
  Therefore,
  \begin{align*}
    \nu_{2}(\mathcal{H}_\theta(\lambda_\theta))
    \geq\,
    \nu_{2}(\mathcal{H}_{\bar\theta}(0)) -
    \|F_{\bar\theta} - F_\theta\|
    -
    \frac{1}{\sigma_s} \|\mathcal{G}\| \|\nabla{f}_\theta(\bar{x})\|.
  \end{align*}
  Hence, $\nu_{2}(\mathcal{H}_\theta(\lambda_\theta)) > 0$ when
  $\frac{1}{\sigma_s}\,
  \| \mathcal{G}\|\, \|\nabla f_{\theta}(\bar{x})\|
  + \|F_\theta {-} F_{\bar\theta}\|
  \,<\, \nu_{2}(F_{\bar\theta})$.
\end{proof}

We may also provide an analogous theorem for the inhomogeneous setting.

\begin{theorem}\label{thm:boundinhom}
  Consider the setting of \Cref{thm:firstresultinhom}.
  Let $\bar\theta$ be a zero duality gap parameter, and let $\bar y$ be the minimizer.
  Let $\theta$ be another parameter for which $\bar{y}$ is a critical point of~\eqref{eq:Qthetaobjinhom}
  and also
  \begin{align*}
      \frac{1}{\tilde\sigma_s}\,
      \| \tilde{\mathcal{G}}\|\, \|\nabla \tilde{f}_{\theta}(\bar{y})\|
      + \|\tilde{F}_\theta {-} \tilde{F}_{\bar\theta}\|
      \,<\, \nu_{1}(\tilde{F}_{\bar\theta})
  \end{align*}
  where $s \!=\! \codim_{\bar{y}}\mathbf{Y}$,
  $\tilde\sigma_s = \sigma_s(\nabla \tilde{g}(\bar{y}))$ is the $s$-th smallest singular value of the Jacobian,
  $\nu_1(\,)$ denotes the smallest eigenvalue,
  and $\| \tilde{\mathcal{G}}\|$ is its operator norm of the linear map $\tilde{\mathcal{G}}(\lambda) \!:=\! \sum_{i=1}^m \lambda_i \tilde{G}^i$,
  Then $(\tilde{Q}_\theta^{\mathrm{obj}})$ has zero duality gap
  and $(\tilde{P}_\theta^{\mathrm{obj}})$ recovers the minimizer.
\end{theorem}
\begin{proof}
  The proof is identical to \Cref{thm:boundhom},
  except that we need lower bound
  $ \nu_{1}(\tilde{\mathcal{H}}_{\bar\theta}(\lambda_\theta)) $
  instead of
  $ \nu_{2}(\mathcal{H}_{\bar\theta}(\lambda_\theta)) $.
  Observe that
  $ \nu_{1}(\tilde{\mathcal{H}}_{\bar\theta}(\lambda_\theta))
  \leq \nu_{2}(\mathcal{H}_{\bar\theta}(\lambda_\theta)) $
  because of Cauchy's interlacing theorem.
\end{proof}

In the special case of the nearest point problem to a quadratic variety we get a more explicit bound.

\begin{corollary}\label{thm:boundnearest}
  Consider the problem~\Etheta.
  Let $\bar{\theta} \!\in\! \mathbf{Y}$ be such that $\ACQ{\mathbf{Y}}{\bar\theta}$ holds.
  Let $\theta\!\in\!\RR^n$ be a point in the normal space of~$\mathbf{Y}$ at~$\bar\theta$
  (i.e., $\theta{-}\bar\theta$ lies in the row space of $\nabla \tilde g(\bar\theta)$)
  such that
  $ \|\theta {-}\bar\theta \| < \tilde\sigma_s/2 \|\tilde{\mathcal{G}}\|.$
  Then~\Etheta has zero duality gap
  and the SDP recovers the minimizer.
\end{corollary}
\begin{proof}
  Observe that $\bar\theta$ is a critical point of~\Etheta since
  $\nabla f_\theta(\bar\theta) = 2 (\bar\theta{-}\theta)\!^\top\!
  = \lambda^\top \nabla \tilde g(\bar\theta)$ for some~$\lambda$.
  The result follows from \Cref{thm:boundinhom} by noticing that
  $\|\nabla \tilde{f}_\theta(\bar{y})\| = 2\|\theta{-}\bar\theta\|$
  and $\tilde{F}_\theta = \tilde{F}_{\bar\theta} = \id_n$.
\end{proof}

\Cref{thm:boundhom,thm:boundinhom,thm:boundnearest} allows us to obtain inner approximations of the SDP-exact region.
In particular, for the family \Etheta the set
\begin{align}\label{eq:boundnearest}
  \bigcup_{\bar\theta \in \mathbf{Y}}\;
  \Bigl\{ \theta \in N_{\mathbf{Y}}(\bar\theta):
    \|\theta -\bar\theta \| \,<\, \frac{\tilde\sigma_s}{2 \|\tilde{\mathcal{G}}\|}
  \Bigr\}
\end{align}
gives such inner approximation,
where $N_{\mathbf{Y}}(\bar\theta)$ denotes the normal space of $\mathbf{Y}$ at~$\bar\theta$.

\begin{example}
  Consider the twisted cubic from \Cref{exmp:twistedcubic}.
  The SDP-exact region corresponds to the dotted region in \Cref{fig:twistedcubic2}.
  Its boundary is defined by a univariate polynomial of degree~8; see \cite[Ex~6.1]{Cifuentes2019}.
  The darker region in \Cref{fig:twistedcubic2} is the inner approximation from~\eqref{eq:boundnearest}.
  \begin{figure}[htb]
    \centering
    \includegraphics[clip,trim=0 {5pt} 0 {20pt},width=270pt]{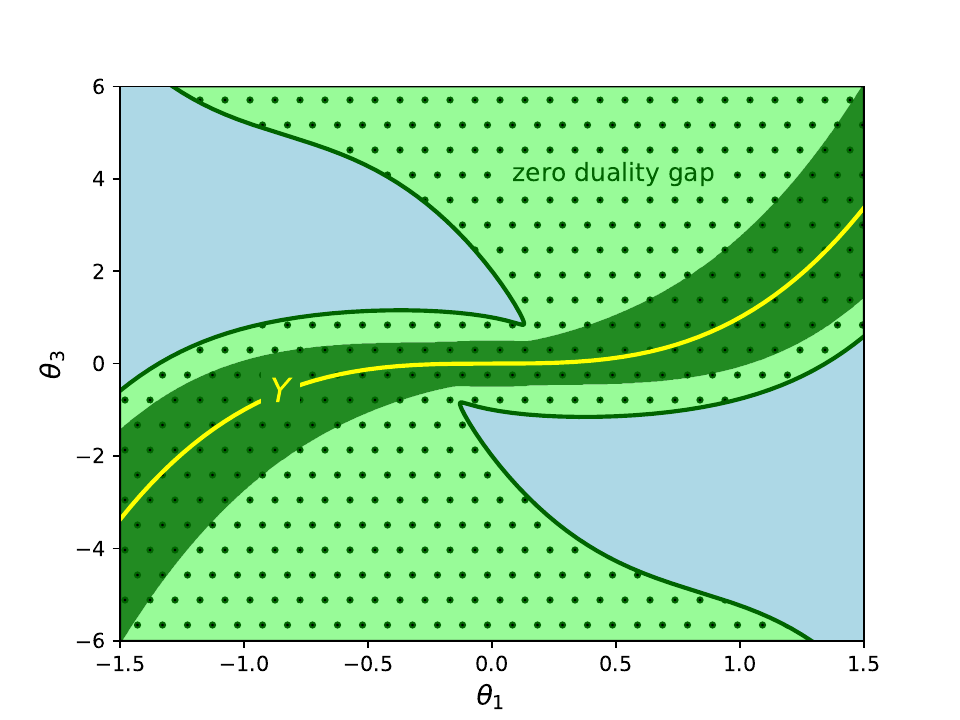}
    \caption{Region of zero duality gap from \Cref{fig:twistedcubic}.
    The darker region is the guaranteed region of zero duality gap (\Cref{thm:boundnearest}).
    }
    \label{fig:twistedcubic2}
  \end{figure}
\end{example}

\section{The General Case}\label{s:main}

We are now ready to consider the general parametrized family of QCQPs \Qtheta,
along with their associated SDPs \Ptheta and \Dtheta.
We assume throughout this section
that $F_\theta, G_\theta^i, c_\theta^i$ are continuous \emph{differentiable} functions of $\theta$,
and that at least one $c_\theta^i \!\neq\! 0$.
Let $\bar\theta \!\in\! \Theta$ be a zero duality gap parameter,
$\bar{x} \!\in\! \RR^N$ be optimal for~\Qthetabar,
and $\bar{\lambda} \!\in\! \RR^m$ be optimal for~$(D_{\bar\theta})$.
We denote $\bar H := \mathcal{H}_{\bar{\theta}}(\bar\lambda)\in\SR^N$ half the Hessian of the Lagrangian at~$\bar\theta$.
Recall that $\bar H\succeq 0$ and $\bar H \bar x = 0$.
We also denote by $\mathbf{X}_\theta:=\{x \!\in\! \RR^N: g_\theta(x) {=} 0 \}$ the (primal) feasible set,
and $\bar{\mathbf{X}} := \mathbf{X}_{\bar\theta}$.

Our first stability result, \Cref{thm:firstresulthom},
relied on two important simplifications:
the feasible region $\mathbf{X}_\theta$ is independent of~$\theta$,
and $\corank F_{\bar\theta} \!=\! 1$.
Our main stability result in this section will relax these assumptions.

Let us first allow the feasible set $\mathbf{X}_\theta$ to depend on~$\theta$.
Under the additional assumption that $\mathbf{X}_\theta$ moves smoothly as a function~$\theta$ (as in \Cref{defn:smoothness}),
it is possible to obtain stability (\Cref{thm:intermediate}).

\begin{definition}\label{defn:smoothness}
  The mapping $\Theta\rightrightarrows \RR^N$, $\theta \mapsto \mathbf{X}_\theta$ is \emph{smooth} nearby $\bar{w} := (\bar\theta,\bar x)$
  if its graph $\mathbf{W} \!:=\! \{(\theta,x): x \!\in\! \mathbf{X}_\theta\}$
  is a smooth manifold nearby $\bar{w}$,
  and $\dim_{\bar{w}} \mathbf{W} \!=\! \dim \Theta \!+\! \dim_{\bar{x}} \bar{\mathbf{X}} $.
\end{definition}

\begin{theorem}\label{thm:intermediate}
  Assume that
  $\bar H$ has corank one,
  $\ACQ{\bar{\mathbf{X}}}{\bar{x}}$ holds,
  and the mapping $\theta \mapsto \mathbf{X}_\theta$ is smooth nearby $(\bar\theta,\bar x)$.
  Then \Qtheta is SDP-stable near~$\theta$
  and \Ptheta recovers the minimizer.
\end{theorem}

We will not prove \Cref{thm:intermediate} directly,
but rather obtain it as a special instance of our main result,
in which we also relax the assumption that $\corank\bar H \!=\! 1$.
Relaxing this assumption is challenging.
We will replace it by two weaker assumptions.

\begin{definition}
  A point $x \in \bar{\mathbf{X}}$ is a \emph{branch point} of~$\bar{\mathbf{X}}$
  with respect to a linear map $\pi: \RR^{N}\to \RR^k$
  if $\ker(\pi) \cap T_{{x}}\bar{\mathbf{X}} \neq 0$,
  where $T_{{x}} \bar{\mathbf{X}} := \ker \nabla {g}_{\bar\theta}({x})$ is the tangent space of~$\bar{\mathbf{X}}$ at~${x}$.
\end{definition}

The next definition is non-standard.

\begin{definition}\label{defn:RS}
  The \emph{restricted Slater} condition holds at~$(\bar x, \bar\lambda)$
  if there exists $\lambda' \!\in\! \RR^m$ such that
  the quadratic function
  $\Psi_{\lambda'}(x):= \sum_i \lambda'_i\, g^i_{\bar\theta}(x)$
  is strictly convex on
  \begin{align*}
    V \,:=\, \{v\in \RR^N: \bar H v \!=\! 0, \bar{x}^\top v \!=\! 0\}
    \,=\, \ker \bar H \cap (\bar x)^\perp,
  \end{align*}
  and also $\nabla \Psi_{\lambda'}(\bar{x}) = 2 \sum_i \lambda_i' G^i_{\bar \theta} \bar{x} = 0$.
\end{definition}

We proceed to present our most general result,
and afterwards we will discuss the new assumptions in detail.

\begin{theorem}[Main result]\label{thm:main}
  Assume that:
  \begin{enumerate}[label=(\roman*)]
    \item \emph{(ACQ)}
      the constraint qualification $\ACQ{\bar{\mathbf{X}}}{\bar{x}}$ holds.
    \item \emph{(smoothness)}
      the mapping $\theta \mapsto \mathbf{X}_\theta$ is smooth nearby $(\bar\theta,\bar x)$.
    \item \emph{(non-branch point)}
      $\bar{x}$ is not a branch point of $\bar{\mathbf{X}}$ with respect to $x \mapsto \bar{H}x$.
    \item \emph{(restricted Slater)}
      the restricted Slater condition holds at~$(\bar x,\bar\lambda)$.
  \end{enumerate}
  Then \Qtheta is SDP-stable near~$\bar{\theta}$
  and \Ptheta recovers the minimizer.
\end{theorem}

\Cref{thm:intermediate} is a special case of \Cref{thm:main},
since if $\corank \bar H \!=\! 1$ then the non-branch point and restricted Slater conditions are satisfied.

\subsection{Discussion of the assumptions}\label{s:generalcase}

\Cref{thm:main} has four assumptions.
The ACQ and smoothness conditions are concerned with the regularity of the feasible set.
ACQ says that the fixed variety~$\bar{\mathbf{X}} = \mathbf{X}_{\bar\theta}$ is smooth nearby~$\bar x$,
while the smoothness condition ensures that the family of varieties $\mathbf{X}_\theta$ behaves well nearby~$\bar\theta$.
We proceed to introduce the remaining two conditions from \Cref{thm:main}.

\subsubsection{Non-branch point}\label{s:branchpoint}

The ACQ property guarantees regularity of the feasible set $\bar{\mathbf{X}}$ nearby~$\bar x$.
In order to guarantee SDP-stability,
we will also need regularity of the image of~$\bar{\mathbf{X}}$ under the linear map $x \mapsto \bar{H} x$.
The next example illustrates the issues we may face when this image is not regular.

\begin{example}\label{exmp:branch}
  Consider the nearest point problem to the plane curve
  $\mathbf{Y}:= \{y \!\in\! \RR^2:  y_2^2 \!=\! y_1^3\}$.
  By introducing the auxiliary variable~$z$
  we can rewrite the problem as the inhomogeneous QCQP
  $ \min \{ \|y\!-\!\theta\|^2 : y_1 \!=\! z^2,\, y_2 \!=\!y_1 z\} $.
  The feasible set of this QCQP is the twisted cubic curve
  $\mathbf{X} := \{(y_1,y_2,z): y_1 {=}z^2,\, y_2 {=} z^3\}$,
  which satisfies ACQ everywhere.
  However, the parameter $\bar\theta \!=\! (0,0)$
  is problematic for the SDP relaxation (e.g., \Pthetabar has multiple solutions).
  The underlying cause is that the optimal solution $\bar y \!=\! (0,0)$ is a singular point of the plane curve~$\mathbf{Y}$.
  As shown in \Cref{fig:cubiccurve2},
  the singular curve~$\mathbf{Y}$ is the projection of $\mathbf{X}$
  under the map $(y_2,y_3,z)\mapsto (y_2,y_3)$.
  This linear map agrees with $x \mapsto \bar H x$ for the point $\bar\theta \!=\! (0,0)$.
  \begin{figure}[htb]
    \centering
    \includegraphics[clip,trim=0 {30pt} 0 {65pt},width=270pt]{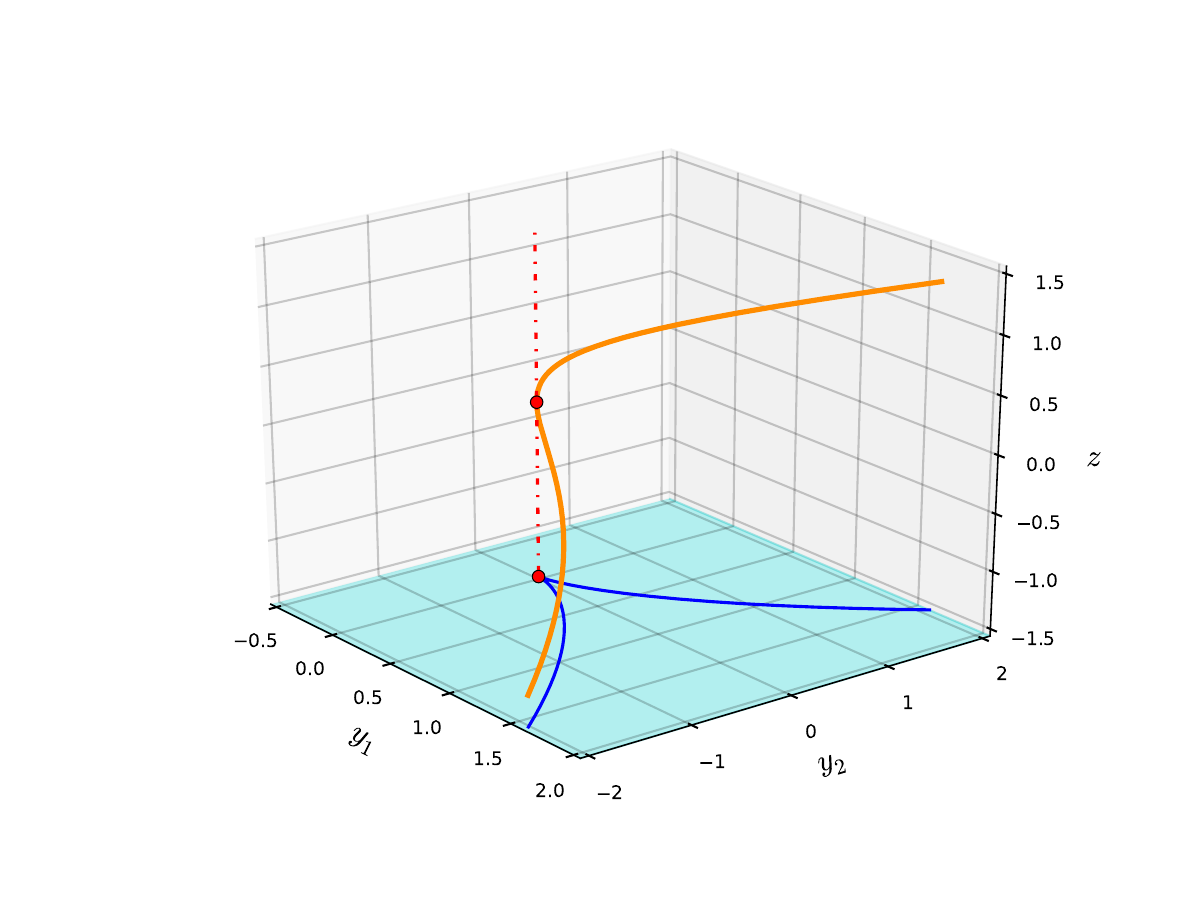}
    \vspace{-10pt}
    \caption{The plane curve $y_2^2\!=\!y_1^3$ is the projection of the twisted cubic. The smooth point $(0,0,0)$ is a branch point of the twisted cubic curve with respect to the projection. It maps to the singular point $(0,0)$ on $y_2^2\!=\!y_1^3$.}
    \label{fig:cubiccurve2}
  \end{figure}
\end{example}

The purpose of the non-branch point assumption is to ensure that $\bar H \bar x$ is a regular point of $\bar H(\bar{\mathbf{X}})$.
More generally, consider a linear map $\pi: \RR^{N} \to \RR^k$ and let $\bar{x}$ be a regular point of~$\bar{\mathbf{X}}$ (i.e., $\ACQ{\mathbf{X}}{\bar x}$ holds).
Then regularity of $\pi(\bar x)$ may only fail when $\bar x$ is a {branch point} of~$\pi$
(i.e., if $\ker(\pi) \cap  T_{\bar{x}}\bar{\mathbf{X}} \!\neq\! 0$).

\begin{example}
  Consider the setup of \Cref{exmp:branch}.
  The reason why the smooth curve $\mathbf{X}$ is projected into the singular curve $\mathbf{Y}$ is that
  $\bar x = (0,0,0)$ is a branch point of $\mathbf{X}$ with respect to $(y_1,y_2,z)\mapsto(y_1,y_2)$.
  This can be observed in \Cref{fig:cubiccurve2},
  by noticing that the tangent line of $\mathbf{X}$ at $\bar x$ is precisely the $z$-axis.
\end{example}

\subsubsection{Restricted Slater}

The previous assumptions (ACQ, smoothness, non-branch point) are all regularity conditions on certain domains (either $\bar{\mathbf{X}}$, $\mathbf{W}$, or $\bar H(\bar{\mathbf{X}})$).
Importantly, all of these regularity assumptions are satisfied for generic problems.
However, they are not enough to guarantee stability of the relaxation,
as illustrated by the following example.

\begin{example}[Non-informative dual]\label{exmp:twistedcubicbad}
  Consider the following homogeneous QCQP:
  \begin{align*}
    \min_{z\in \RR^3, y\in \RR^3}\;
    \|y-\theta z_0\|^2,
    \quad\text{s.t.}\quad
    z_0^2 = 1,\quad
    z_1^2+z_2^2 = 1 , \quad
    (\begin{smallmatrix}z_1 & z_2\end{smallmatrix})
    \bigl(\begin{smallmatrix}z_0 & y_1 & y_2\\y_1& y_2 & y_3\end{smallmatrix}\bigr)
    \!= (\begin{smallmatrix}0 & 0 & 0\end{smallmatrix}).
  \end{align*}
  This is in fact computing the nearest point to the twisted cubic.
  Indeed, assuming that $z_0 \!=\! 1$,
  the matrix
  $\bigl(\begin{smallmatrix}z_0 & y_1 & y_2\\y_1& y_2 & y_3\end{smallmatrix}\bigr)$
  is rank deficient if and only if $(y_1,y_2,y_3)$ lies on the twisted cubic.
  A simple calculation shows that the feasible set is regular everywhere (ACQ holds).
  The feasible set is independent of~$\theta$, so the smoothness assumption also holds.
  As the projection $(z,y)\mapsto y$ leads to the twisted cubic,
  which is smooth, the non-branch point condition also holds.
  Nonetheless, we claim that $\val(D_\theta) \!=\! 0$ for any~$\theta$,
  which means that there is zero duality gap only if $\theta$ lies on the twisted cubic.
  The dual problem~\Dtheta involves five variables
  $\lambda_0,\lambda_1,\dots,\lambda_4$ corresponding to the five constraints,
  and has the form:
  \begin{align*}
    \max_{\lambda\in \RR^5}\;\; -\lambda_0 - \lambda_1
    \qquad\text{s.t.}\qquad
    \biggl(\begin{smallmatrix}
        \lambda_0 + \|\theta\|^2 & b(\lambda)\!^\top & -\theta^\top\\
        b(\lambda) & \lambda_1\! \id_2 & B(\lambda)\!^\top \\
        -\theta & B(\lambda) & \id_3
    \end{smallmatrix}\biggr)
    \succeq 0,
  \end{align*}
  where $b(\lambda) \!\in\! \RR^{2}, B(\lambda) \!\in\! \RR^{3\times 2}$ depend linearly on~$\lambda$.
  Observe that the cost $-\lambda_0 {-} \lambda_1$ is non-positive for any feasible~$\lambda$.
  Indeed, the constraint
  $ \lambda_1\! \id \succeq 0 $
  implies $\lambda_1 \!\geq\! 0$,
  while
  $ \bigl(\begin{smallmatrix}
      \lambda_0 + \|\theta\|^2 & -\theta^\top\\[-1pt] -\theta & \id
  \end{smallmatrix}\bigr) \succeq 0 $
  implies $\lambda_0 \!\geq\! 0$.
  It follows that $\lambda \!=\! 0$ is dual optimal,
  so $\val(D_\theta) \!=\! 0$ for any~$\theta$.
  Note that in this example, $\bar{H} = F_{\bar \theta}$ which has corank $> 1$ since the objective function does not contain $z_1, z_2$.
\end{example}

The previous example shows that in the absence of the corank-one assumption on $\bar H$ the problem may be unstable, even under natural regularity assumptions.
We will show in the next section that the restricted Slater assumption
allows us to pass from $\bar \lambda$ to a different dual optimal solution
$\lambda_t$ such that $\mathcal{H}_{\bar \theta}(\lambda_t)$ has corank one.
This helps to restore stability nearby~$\bar \theta$.

Unlike the previous regularity conditions,
restricted Slater is a semialgebraic condition which is not satisfied generically.
Nonetheless, restricted Slater can be verified efficiently
as it corresponds to the strict feasibility of an SDP
(find $\lambda'$ s.t.\ $\sum_i \lambda'_i G^i_{\bar\theta}\bar{x}=0$, $(\sum_i \lambda'_i G^i_{\bar\theta})|_V \succ 0$).

\begin{example}
  Let us see that restricted Slater is not satisfied in \Cref{exmp:twistedcubicbad}.
  Recall that ACQ holds,
  i.e., $\rank \nabla g_{\bar\theta}(\bar x) = \codim \bar{\mathbf{X}} = 5$.
  The $\lambda'$ from \Cref{defn:RS} must satisfy
  $(\lambda')\!^\top \nabla g_{\bar\theta}(\bar x) = 0$.
  Since $\nabla g_{\bar\theta}(\bar x)$ has full rank,
  then $\lambda'\!=\!0$.
  Hence, restricted Slater does not hold.
\end{example}

\subsection{Proof of the main theorem}
We first see how the vector $\lambda'$ from \Cref{defn:RS} gives us a \emph{direction} along which we can perturb $\bar \lambda$
to obtain a new Hessian with corank one.

\begin{lemma}
  \label{thm:perturbation}
  Let $(\bar{x},\bar{\lambda})$ be primal/dual optimal at~$\bar\theta$,
  and let $\lambda'$ be as in \Cref{defn:RS}.
  Then there is an $\epsilon>0$ such that $\lambda_t:= \bar\lambda + t \lambda'$ is dual optimal and $\corank\mathcal{H}_{\bar\theta}(\lambda_t)=1$ for any $0<t<\epsilon$.
\end{lemma}

The proof of \Cref{thm:perturbation} relies on the following well-known lemma.

\begin{lemma}[Finsler~\cite{Finsler1936}]\label{thm:finsler}
  Let $A, B\!\in\! \SR^n$, $A\!\succeq\! 0$, be such that $v^\top B v \!>\! 0 $ for every nonzero $v$ with $A v \!=\! 0$.
  Then there is an $\epsilon\!>\!0$ such that $A \!+\! t B \succ 0$ for any $0<t<\epsilon$.
\end{lemma}

\begin{proof}[Proof of \Cref{thm:perturbation}]
  Since $(\bar{x},\bar{\lambda})$ is primal/dual optimal, it satisfies the three conditions in \Cref{lem:zerodualitygap}.
  We need to show that $(\bar{x},\lambda_t)$ also satisfies these conditions, and that $\corank\mathcal{H}_{\bar\theta}(\lambda_t) \!=\! 1$.
  Let $A:= \mathcal{H}_{\bar\theta}(\bar\lambda)\succeq 0$ and $B:=\sum_i \lambda'_i G^i_{\bar\theta}$.
  Observe that $A\bar x \!=\! 0$ since $\bar\lambda \in \Lambda_{\bar\theta}(\bar{x})$,
  and $B\bar x \!=\! 0$ by definition of~$\lambda'$.
  Since $\mathcal{H}_{\bar\theta}(\lambda_t) \!=\! A {+} t B$
  then
  $\mathcal{H}_{\bar\theta}(\lambda_t)\bar x \!=\! 0$,
  so $\lambda_t\in \Lambda_{\bar\theta}(\bar{x})$.
  It remains to show that $A{+}t B\succeq 0$ and has corank-one.
  We may assume without loss of generality that $\bar{x} = (1,0^{N-1})$.
  Then
  $A = \bigl(\begin{smallmatrix}0&0\\0&A'\end{smallmatrix}\bigr)$,
  $B = \bigl(\begin{smallmatrix}0&0\\0&B'\end{smallmatrix}\bigr)$,
  where $A',B'\in \SR^{N-1}$, $A'\succeq 0$.
  Note that $v^\top B' v > 0 $ for every nonzero $v\in\RR^{N-1}$ with $A' v = 0$.
  From \Cref{thm:finsler} we know that $A' {+} t B' \succ 0$ for all $0<t<\epsilon$.
  Therefore, $A {+} t B \succeq 0$ and has corank one for $0<t<\epsilon$, as wanted.
\end{proof}

Recall that \Cref{prop:stability} shows that $\corank\mathcal{H}_{\bar \theta}(\bar \lambda)\!=\!1$ implies SDP-stability near $\bar \theta$,
as long as the Lagrange multiplier mapping satisfies weak continuity.
\Cref{thm:perturbation} allows us to find a dual optimal solution $\lambda_t$
for which $\corank\mathcal{H}_{\bar \theta}(\lambda_t) \!=\! 1$.
In order to use \Cref{prop:stability} and conclude SDP-stability,
it remains to see that $(\bar x, \lambda_t)$ satisfies weak continuity, which can be obtained via a \emph{stronger} continuity requirement on the original pair~$(\bar x,\bar\lambda)$.
We first recall a well-studied notion of continuity for set-valued-mappings.
We refer to~\cite{Rockafellar2009,Aubin2009} for a detailed introduction to set-valued-mappings.

\begin{definition}[Painlev\'e-Kuratowski continuity]\label{defn:hemicontinuous}
  Let $\mathfrak{F}: \Theta \rightrightarrows \RR^k$ be a {set-valued mapping},
  and assume that each $\mathfrak{F}(\theta)\subset \RR^k$ is nonempty.
  A \emph{selection} of $\mathfrak{F}$ is an assignment $\ell_\theta\in \mathfrak{F}(\theta)$ for each $\theta\in\Theta$.
  The \emph{inner limit} of $\mathfrak{F}$ at $\bar\theta$
  consists of all limits of selections $\{\ell_\theta\}_\theta$, i.e.,
  \begin{align*}
    \liminf_{\theta \to \bar\theta} \mathfrak{F}(\theta)
    &:=\{ \ell\in\RR^k : \exists \ell_\theta \in \mathfrak{F}(\theta) \text{ s.t. } \ell_\theta \xrightarrow{\theta\to\bar\theta} \ell \},
  \end{align*}
  The \emph{outer limit} of $\mathfrak{F}$ at $\bar\theta$
  consists of all cluster points of selections $\{\ell_\theta\}_\theta$, i.e.,
  \begin{align*}
    \limsup_{\theta \to \bar\theta} \mathfrak{F}(\theta)
    &:=\{ \ell\in\RR^k : \exists \theta_i \xrightarrow{i\to\infty} \bar\theta,\; \exists \ell_i \in \mathfrak{F}(\theta_i) \text{ s.t. } \ell_i \xrightarrow{i\to\infty} \ell \}.
  \end{align*}

  The inner and outer limits are always closed sets that sandwich the closure of~$\mathfrak{F}({\bar\theta})$:
  \begin{align*}
    \liminf_{\theta \to \bar\theta} \mathfrak{F}(\theta)
    \subset \closure( \mathfrak{F}({\bar\theta}) )
    \subset \limsup_{\theta \to \bar\theta} \mathfrak{F}(\theta).
  \end{align*}
  $\mathfrak{F}$ is (Painlev\'e-Kuratowski) \emph{continuous}%
  \footnote{
    Although other notions of (set-valued-mapping) continuity exist, they agree for the case of compact valued mappings~\cite{Rockafellar2009}.
    Since the analysis done in this paper is local, we may always restrict the range to some closed ball.
    Hence, we may ignore this distinction in this paper.
  }
  at $\bar\theta$ if
  $
    \mathfrak{F}({\bar\theta})
    = \liminf_{\theta \to \bar\theta} \mathfrak{F}(\theta)
    = \limsup_{\theta \to \bar\theta} \mathfrak{F}(\theta).
  $
\end{definition}

\begin{remark}
  When $\mathfrak{F}$ is defined by continuous functions,
  such as $\mathfrak{L}$, then the equation
  $ \mathfrak{F}({\bar\theta}) = \limsup_{\theta \to \bar\theta} \mathfrak{F}(\theta) $
  always holds \cite[Ex~5.8]{Rockafellar2009}.
  Consequently, in this paper we will focus our attention only on the inner limit.
\end{remark}

\begin{remark}
  Note that weak continuity is simply that $\bar\ell = (\bar x, \bar \lambda)
  \in \liminf_{\theta\to\bar\theta} \mathfrak{L}(\theta)$.
\end{remark}

\begin{example}
  Consider the mapping
  \begin{align*}
    \mathfrak{F}: \RR \rightrightarrows \RR, \qquad
    \theta \mapsto
    \begin{cases}\{0\}, &\text{if }\theta< 0 \\ [-1,1], &\text{if }\theta\geq 0\end{cases}
  \end{align*}
  This mapping is continuous at any $\theta\neq 0$.
  Observe that $\liminf_{\theta \to 0}\mathfrak{F}(\theta) = \{0\}$ and $\limsup_{\theta \to 0}\mathfrak{F}(\theta) = [-1,1]$.
  Thus $\mathfrak{F}$ is not continuous at~$0$.
\end{example}

\begin{definition}
  \label{defn:strongcontinuity}
  The Lagrange multiplier mapping $\mathfrak{L}$ is \emph{strongly continuous}
  at a pair $\bar\ell = (\bar x, \bar\lambda) \in \mathfrak{L}(\bar\theta)$
  if there exists a closed neighborhood $U\ni \bar{\ell}$ such that
  $\mathfrak{L}(\bar{\theta})\cap U \subset \liminf_{\theta\to\bar\theta} \mathfrak{L}(\theta)$,
  or equivalently, such that
  the mapping $\theta \mapsto \mathfrak{L}(\theta)\cap U$ is continuous at~$\bar\theta$.
\end{definition}

The next proposition shows that the three regularity conditions (ACQ, smoothness, non-branch point) guarantee strong continuity of the multipliers.
This proposition can be extended to arbitrary nonlinear programs.
The proof relies on the implicit function theorem, but it also requires some technical definitions from variational analysis.
Hence we postpone the proof to \Cref{s:implicitfunction}.

\begin{proposition}
  \label{thm:aubin}
  Let $\bar{x}$ be a critical point of~\Qthetabar and $\bar\lambda\in\Lambda_{\bar\theta}(\bar{x})$.
  Assume that $\ACQ{\bar{\mathbf{X}}}{\bar{x}}$ holds,
  the mapping $\theta \mapsto \mathbf{X}_\theta$ is smooth nearby $(\bar\theta,\bar x)$,
  and $\bar{x}$ is not a branch point of $\bar{\mathbf{X}}$ with respect to $x \mapsto \bar{H}x$.
  Then strong continuity holds at~$(\bar x, \bar\lambda)$.
\end{proposition}

The proof of \Cref{thm:main} is now completed by the following theorem,
which is the analog of~\Cref{prop:stability}.

\begin{proposition}
  \label{thm:strongcontinuity}
  Let $(\bar{x},\bar{\lambda})$ be primal/dual optimal at~$\bar\theta$,
  such that restricted Slater and strong continuity hold.
  Then \Qtheta is SDP-stable near~$\bar\theta$ and \Ptheta recovers the minimizer.
\end{proposition}

\begin{proof}
  Let $U\ni \bar \ell$ be as in \Cref{defn:strongcontinuity}.
  By \Cref{thm:perturbation}, there is a dual optimal solution $\lambda_t$,
  arbitrarily close to $\bar\lambda$, such that $\mathcal{H}_{\bar\theta}(\lambda_t)$ has
  corank one. Thus, we may assume that $\ell_t:=(\bar{x},\lambda_t) \in U$. We already
  saw that $\lambda_t$ is a Lagrange multiplier of $\bar{x}$. Therefore,
  $\ell_t$
  belongs to $\mathfrak{L}(\bar \theta) \cap U\subset  \liminf_{\theta\to\bar\theta} \mathfrak{L}(\theta)$.
  Since $\corank \mathcal{H}_{\bar\theta}(\lambda_t) \!=\! 1$
  and $\ell_t$ satisfies weak-continuity,
  the theorem follows from \Cref{prop:stability}.
\end{proof}

\subsection{The inhomogeneous version of the main theorem}

Many applications are better suited to an inhomogeneous version of \Cref{thm:main}, which we now derive.

\begin{theorem} \label{thm:maininhom}
  Consider the inhomogeneous family of QCQPs:
  \begin{align}\label{eq:Qthetatilde}
    \tag{$\tilde{Q}_\theta$}
    \min_{y \in \RR^n}\;
    \tilde{f}_\theta(y)
    \!:=\! y^\top \tilde{F}_{\theta} y
    \!+\! 2 \tilde{l}_\theta^{\,\top} y \!+\! \tilde{c}_\theta,
    \quad\text{s.t.}\quad
    \tilde{g}_\theta^i(y)
    \!:=\! y^\top \tilde{G}_\theta^i y \!+\!
    2 (\tilde{l}_\theta^i)^{\!\top} y \!+\! \tilde{c}_\theta^i \!=\!0,
    \;\; i \!=\! 1,\ldots,m.
  \end{align}
  Let $\bar\theta$ be a zero duality gap parameter,
  $(\bar y, \bar \mu)$ be the primal/dual optimal solutions at~$\bar\theta$, and
  \begin{align*}
    \tilde{\mathcal{H}}_{\theta}(\mu) :=
    \tilde{F}_\theta + \sum\nolimits_i \mu_i \tilde{G}_\theta^i \in \SR^{n},
    \quad\;\;
    \tilde{H} := \tilde{\mathcal{H}}_{\bar{\theta}}(\bar\mu),
    \quad\;\;
    \mathbf{Y}_\theta:=\{y \!\in\! \RR^n: \tilde{g}_\theta(y) {=} 0 \},
    \quad\;\;
    \bar{\mathbf{Y}}:= \mathbf{Y}_{\bar\theta}.
  \end{align*}
  Assume that the following conditions hold:
  \begin{enumerate}[label=(\roman*)]
    \item \emph{(ACQ)}
      the constraint qualification $\ACQ{\bar{\mathbf{Y}}}{\bar{y}}$ holds.
    \item \emph{(smoothness)}
      the mapping $\theta \mapsto \mathbf{Y}_\theta$ is smooth nearby $(\bar\theta,\bar y)$.
    \item \emph{(non-branch point)}
      $\bar{y}$ is not a branch point of $\bar{\mathbf{Y}}$ with respect to $y \mapsto \tilde{H}y$.
    \item \emph{(restricted Slater)}
      There exists $\mu' \!\in\! \RR^m$ such that
      the quadratic function $\tilde\Psi_{\mu'}(y) \!:= \sum_i \mu_i' \tilde{g}^i_{\bar\theta}(y)$ is strictly convex on~$U \!:=\! \ker \tilde H$,
      and also $\nabla \tilde\Psi_{\mu'}(\bar{y}) \!=\! 0$.
  \end{enumerate}
  Then \Qthetatilde is SDP-stable near~$\bar{\theta}$
  and the SDP recovers the minimizer.
\end{theorem}

\begin{proof}
  We first argue that we can reduce the problem to the case $\bar y \!=\! 0$.
  If this is not the case, then we may consider the change of variables $w := y {-} \bar y$.
  It is easy to see that that if the conditions from \Cref{thm:maininhom} were satisfied for the original problem,
  then they are also satisfied for the modified problem.
  Hence, we may assume that $\bar y = 0$.

  The homogenized problem has primal variables $x \!=\! (z_0, y)$
  and dual variables $\lambda \!=\! (\lambda_0,\mu)$.
  Note that $\bar x \!=\! (1,0)$ since $\bar y \!=\! 0$.
  The homogeneous and inhomogeneous problems are related as follows:
  \begin{gather*}
    \nabla g_{\bar\theta}(\bar x) =
    \biggl(\begin{matrix} 2 & 0 \\ -\nabla \tilde{g}_{\bar\theta}(\bar y) \bar y & \nabla \tilde{g}_{\bar\theta}(\bar y) \end{matrix}\biggr),
    \qquad\qquad
    \bar{H} =
    \biggl(\begin{matrix} 0 & 0 \\ 0 & \tilde{H} \end{matrix}\biggr),
    \\
    \nabla \Psi_{\lambda}(\bar x) =
    \biggl(\begin{matrix} 2\lambda_0  \!-\! \nabla \tilde{\Psi}_{\mu}(\bar y) \bar y & \nabla \tilde{\Psi}_{\mu}(\bar y) \end{matrix}\biggr),
    \qquad
    \nabla^2 \Psi_{\lambda} =
    \biggl(\begin{matrix} 2\lambda_0 & *\\ * & \nabla^2 \tilde{\Psi}_{\mu} \end{matrix}\biggr).
  \end{gather*}
  We proceed to show that \Qthetatilde satisfies the conditions \emph{(i)---(iv)} in \Cref{thm:maininhom}
  if and only if its homogenization \Qtheta satisfies the conditions in \Cref{thm:main}.
  \begin{enumerate}[label=(\roman*)]
    \item
      $\rank\nabla g_{\bar \theta}(\bar{x}) = \rank\nabla \tilde{g}_{\bar \theta}(\bar{y}) {+}1$,
      $\dim_{\bar{x}}\bar{\mathbf{X}} = \dim_{\bar{y}}\bar{\mathbf{Y}}$,
      $\codim_{\bar{x}}\bar{\mathbf{X}} = \codim_{\bar{y}}\bar{\mathbf{Y}} {+}1$.

    \item
      $\mathbf{X}_\theta = \{(z_0,y): z_0 \!=\! \pm 1,\, z_0 y \!\in\! \mathbf{Y}_\theta\}$
      consists of two disjoint copies of~$\mathbf{Y}_\theta$.
    \item
      Notice that
      $ \ker \nabla g_{\bar\theta}(\bar x)
      = \{0\} \times \ker \nabla \tilde{g}_{\bar\theta}(\bar y) $
      and that
      $ \ker \bar{H} = \RR \times \ker \tilde{H} $.
      The result follows from the equation
      $\ker \nabla g_{\bar \theta}(\bar x) \cap \ker \bar{H}
      = \{0\} \!\times\! (\ker \nabla \tilde{g}_{\bar \theta} (\bar y) \cap \ker \tilde{H})$.

    \item
      Let $\mathring{\mu} := (0,\mu)$ and $\mathring{u} := (0,u)$ denote the vectors obtained by prepending a zero.
      Note that
      \begin{align*}
        \qquad\nabla{\Psi}_{\mathring\mu}(\bar x) \!=\! 0
        \iff\! \nabla\tilde{\Psi}_{\mu}(\bar y) \!=\! 0,
        \quad
        \ker \bar{H} \cap (\bar x)^\perp\! = \{0\} \!\!\times\! \ker \tilde{H},
        \quad
        \mathring{u}^\top \nabla^2{\Psi}_{\mathring\mu}\, \mathring{u}=
        u^\top \nabla^2\tilde{\Psi}_{\mu} u.
      \end{align*}
      Let $\mu' \!\in\! \RR^m$ such that $\nabla \tilde\Psi_\mu(\bar{y}) \!=\! 0$
      and also $u^\top \nabla^2\tilde\Psi_{\mu'}\, u \!>\! 0$ for all nonzero $u \!\in\! \ker \tilde H$.
      By the above equations we have that
      $\lambda' \!:=\! \mathring\mu'$ satisfies $\nabla \Psi_{\lambda'}(\bar{x}) \!=\! 0$
      and also $v^\top \nabla^2\Psi_{\lambda'}\, v \!>\! 0$ for all nonzero $v \in \ker \tilde H \cap (\bar x)^\perp$.
      The converse implication is similar.
      \qedhere
  \end{enumerate}
\end{proof}

We conclude this section with an illustration of \Cref{thm:maininhom} on
yet another QCQP on the twisted cubic, but this time, the objective function is nonconvex.

\begin{example}
  Given $\theta \!\in\! \RR$,
  consider minimizing
  $y_3^2 {+} y_1 y_3 {-} 2 \theta y_1^2 {-} 2 y_3$
  on the cubic curve defined by $y_3 \!=\! y_1^3$.
  We may rewrite the problem as a QCQP
  by introducing the auxiliary variable $y_2 \!=\! y_1^2$:
  \begin{align}
    \min_{y_1,y_2,y_3 \in \RR} \quad
    y_3^2 {+} y_1 y_3 {-} 2 \theta y_1^2 {-} 2 y_3,
    \quad\text{ s.t. }\quad
    y_2 \!-\! y_1^2 = y_3 \!-\! y_1 y_2 = y_2^2 \!-\! y_1 y_3 = 0.
  \end{align}
  Consider the nominal parameter $\bar\theta \!=\! 1$,
  at which we have
  $\val(Q_{\bar\theta}) \!=\! \val(D_{\bar\theta}) \!=\! -2$.
  The optimal solutions are $\bar{y}=(1,1,1)$ and $\bar\mu = (-2,0,1)$.
  We claim that the assumptions from \Cref{thm:maininhom} hold,
  and hence the problem is SDP-stable nearby~$\bar\theta$.

  The variety~$\bar{\mathbf{Y}}$ is the twisted cubic,
  so ACQ holds everywhere.
  The smoothness assumption also holds as the constraints are independent of~$\theta$.
  We proceed to the non-branch point condition.
  Note that
  \begin{align*}
    \nabla \tilde{g}(\bar{y}) =
    \left(\begin{smallmatrix}
      -2 & 1 & 0 \\
      -1 & -1 & 1 \\
      -1 & 2 & -1
    \end{smallmatrix}\right),
    \qquad
    \tilde{L}_{\bar\theta}(y,\bar\mu)
    = (y_2{-}1)^2 + (y_3{-}1)^2 - 2,
    \qquad
    \tilde H =
    \nabla^2 \tilde{L}_{\bar\theta} =
    \biggl(\begin{smallmatrix}0&0&0\\0&1&0\\0&0&1\end{smallmatrix}\biggr).
  \end{align*}
  Then $\ker \nabla \tilde{g}(\bar{y})$ is spanned by $(1,2,3)$,
  and $\ker \tilde H$ by $(1,0,0)$,
  so $\bar y$ is not a branch point.
  Finally, let us see that restricted Slater holds with $\mu' \!=\! (-1,1,1)$.
  The corresponding quadratic function is
  $
  \tilde\Psi_{\mu'}(y) := y_1^2{-}y_1 y_2{+}y_2^2{-}y_1 y_3{-}y_2{+}y_3
  $
  and one can check that
  $\nabla \tilde\Psi_{\mu'}(\bar y) \!=\! 0$.
  Recall that $U := \ker \tilde H$ is the line spanned by $(1,0,0)$.
  So the restriction of $\tilde\Psi_{\mu'}$ to $U$ is the univariate function $y_1 \mapsto y_1^2$, which is strictly convex.
  We have verified the four conditions needed in \Cref{thm:maininhom}.
\end{example}

\section{Applications}\label{s:applications}
We now apply our results to an array of applications.
Recall that the motivation for this paper was to provide a theoretical understanding of the observation that many statistical estimation problems exhibit zero duality gap under low noise. Our results
provide a uniform framework for understanding these, and other stability results.

\subsection{Estimation problems with a strictly convex objective}\label{s:estimationeasy}

We first consider two nearest point problems to which we apply \Cref{thm:firstresultnearest}. This requires checking the ACQ property for which
we rely on the following well-known fact
(see e.g., \cite[\S14]{Harris2013}).

\begin{lemma}\label{thm:ACQradical}
  Let $g\subset\RR[x]$ be a polynomial system with variety~$\mathbf{X}$.
  If the ideal $\langle g\rangle$ is radical then ACQ holds for each smooth point of~$\mathbf{X}$.
\end{lemma}

\begin{example}[Triangulation]\label{exmp:triangulation}
  In computer vision we represent
  3D points by vectors $z = (z_1 , z_2 , z_3, 1) \in \RR^4$,
  2D points (images) by vectors $u = (u_1, u_2) \in \RR^2$,
  and cameras by matrices $P \in \RR^{3\times 4}$.
  In the triangulation problem we are given
  $\ell$~cameras $P_j \in \RR^{3 \times 4}$
  and noisy images $\hat{u}_j\in \RR^2$ of an unknown 3D point $z\in\RR^4$,
  and the goal is to recover~$z$.
  Assuming i.i.d.\ Gaussian noise, the MLE is given by the nearest point problem:
  \begin{align*}
    \min_{u\in \mathbf{U}}\quad &\sum_{j=1}^\ell\|u_j-\hat{u}_j\|^2,
    \quad\text{where}\quad
    \mathbf{U} := \{ u \in (\RR^2)^\ell: \exists z\in \PP^3
    \text{ s.t. } u_{j} = \Pi P_j {z}\; \text{ for } 1\leq j\leq \ell\},
  \end{align*}
  where $\Pi : \RR^3\to \RR^2$, $(y_1,y_2,y_3) \mapsto (y_1/y_3,y_2/y_3)$ is the
  dehomogenization map.
  The variety $\mathbf{U}$~is known as the \emph{multiview variety}.
  If either $\ell\!=\!2$,
  or $\ell\!\geq\! 4$ and the camera centers are not coplanar,
  then
  $$
  \mathbf{U} = \{u\in (\RR^2)^\ell: \tilde{g}_{ij} (u_i,u_j) = 0, \;1\leq i<j\leq \ell\},
  $$
  where $\tilde{g}_{ij}$ are quadratic equations known as the epipolar constraints~\cite{Heyden1997}.
  This description of~$\mathbf{U}$ gives a QCQP.
  The epipolar equations define a radical ideal~\cite{Heyden1997},
  so ACQ holds at each smooth point (\Cref{thm:ACQradical}).
  In particular, ACQ holds generically on the variety.
  Then, by \Cref{thm:firstresultnearest},
  the SDP relaxation of this QCQP solves the problem exactly (generically) under small noise.

  The above SDP relaxation was considered in~\cite{Aholt2012},
  where they also showed exactness under low noise.
\end{example}

\begin{example}[Tensor PCA]
  Consider vectors $\{v^j \!\in\! \RR^{n_j} \}_{j=1}^\ell$,
  and let $\theta \in \RR^{n_1\times\cdots\times n_\ell}$ be the tensor with entries
  $\theta_{i_1 i_2 \dots i_\ell} = v^1_{i_1} v^2_{i_2} \dots v^\ell_{i_\ell} + \zeta_{i_1 i_2 \dots i_\ell}$,
  where $\zeta_{i_1 i_2 \dots i_\ell}$ are random i.i.d.\ Gaussian variables.
  Hence, $\theta$ is a noisy estimate of the \emph{rank-one tensor}
  $v^1\!\otimes\! \cdots \!\otimes\! v^\ell$.
  The {tensor PCA problem},
  also known as rank-one tensor approximation,
  consists of recovering the vectors $\{v^j\}_j$ from the tensor~$\theta$.
  The maximum likelihood estimator is obtained by solving the nearest point problem:
  \begin{align}\label{eq:tensor}
    \min_{y\in \mathbf{Y}}\quad &\|y-\theta\|^2,
    \quad \text{ where }\quad
    \mathbf{Y} := \{ y \in \RR^{n_1\times\cdots\times n_\ell}: \rank y = 1 \}.
  \end{align}
  This is a QCQP since $\mathbf{Y}$ is the quadratic {\em Segre variety}
  \cite[\S9]{Harris2013},
  cut out by the $2 \!\times\! 2$ minors of the matrix flattenings of~$y$.
  This ideal of minors is radical,
  so by \Cref{thm:ACQradical} we have that ACQ holds everywhere except at the origin (the only singular point).
  Therefore, by \Cref{thm:firstresultnearest},
  the problem is SDP-stable nearby any nonzero $\bar\theta \in Y$.
  In other words,
  the problem~\eqref{eq:tensor} is solved exactly by its SDP relaxation in the low noise regime.
  The same holds for symmetric tensors,
  in which case the associated variety is the \emph{Veronese},
  which is also defined by a radical quadratic ideal.

  It is possible to derive an equivalent QCQP formulation of~\eqref{eq:tensor} involving less variables (so the SDP is cheaper).
  The idea is to eliminate the last component $v^\ell$
  from the tensor $y = v^1\!\otimes\! \cdots \!\otimes\! v^\ell$,
  obtaining a new problem in terms of the unit norm tensor $x$ proportional to
  $v^1\!\otimes\! \cdots \!\otimes\! v^{\ell-1}$.
  This leads to the homogeneous QCQP:
  \begin{align}\label{eq:tensor2}
  \min_{x \in \mathbf{X}}\quad
  \frac{1}{2}\, \sum_{k=1}^{n_\ell}\sum_{I\neq J} (x_I \theta_{J,k} - x_J \theta_{I,k})^2,
  \quad\text{ where }\quad
  \textbf{X} = \{ x \in \RR^{n_1\times \dots \times n_{\ell-1}}\!:
  \rank x \!=\! 1,\; \|x\|^2 \!=\! 1\},
  \end{align}
  and where $I,J$ range over all tuples in $[n_1] \times \dots \times [n_{\ell-1}]$.
  One can check that
  the assumptions of \Cref{thm:firstresulthom} are satisfied,
  so the SDP relaxation of \eqref{eq:tensor2} is also exact under low noise.

  SDP relaxations for the Tensor PCA problem were first studied in~\cite{Nie2014},
  where they introduced the idea of eliminating one component for efficiency.
  Their derivation relies on a parametrization of $\textbf{X}$,
  but the SDP is the same as the one obtained with the implicit description.
  No exactness results were known before.
\end{example}

We now show applications of \Cref{thm:firstresultinhom}.
Note that this theorem has three assumptions:
at the nominal parameter $\bar \theta$,
the objective is strictly convex,
ACQ holds for the minimizer,
which is also the global minimum of the objective function.
Since in the problems below the objective function is a squared loss function, and since in the noiseless case the objective value is zero, the last condition is always satisfied.
Thus, we will only check strict convexity and ACQ.

\begin{example}[$SO(d)$ synchronization]\label{exmp:rotationsync}
  Consider $n{+}1$ objects in $\RR^d$,
  whose orientation is described by matrices $R_i \!\in\! SO(d)$,
  where $SO(d)$ is the \emph{special orthogonal group}.
  Let $G=(V,E)$ be a graph with vertex set $V=\{0,\dots,n\}$.
  In the $SO(d)$ synchronization problem we are given noisy estimates
  $\hat{R}_{ij} \approx R_j R_i^\top$ of the relative orientation among pairs $ij \in E$,
  and the goal is to recover the matrices~$R_i$.
  Under an appropriate noise model~\cite{Rosen2016},
  the MLE is given by the following least-squares problem:
  \begin{align}\label{eq:SOd}
    \min_{R_1,\dots,R_n\in SO(d)}\,
    \sum_{ij\in E}\|R_j \!-\! \hat{R}_{ij}R_i\|_F^2, \quad\;\;
    SO(d) := \{R \!\in\! \RR^{d\times d}: R^\top R \!=\! \id_d,\; \det(R) \!=\! 1 \},
  \end{align}
  where we assume that $R_0 := \id_n$ is the reference point.
  This problem is parametrized by $\theta = (\hat{R}_{ij})_{ij\in E}$.
  To obtain a QCQP, we can replace $SO(d)$ by the \emph{orthogonal group}~$O(d)$:
  \begin{align}\label{eq:Od}
    \min_{R_1,\dots,R_n\in O(d)}\; \sum_{ij\in E}\|R_j - \hat{R}_{ij}R_i\|_F^2, \qquad O(d) := \{R \in \RR^{d\times d}: R^\top R = \id_d\}.
  \end{align}
  Note that~\eqref{eq:SOd} and~\eqref{eq:Od} have the same minimizer in the low noise regime,
  as $SO(d)^n$ is a connected component of $O(d)^n$.
  Consider the SDP relaxation of the QCQP~\eqref{eq:Od}.
  The objective function is strictly convex by \Cref{lem:connectedgraph} below,
  and ACQ is satisfied everywhere since the variety is smooth and the ideal is radical (\Cref{thm:ACQradical}).
  Thus problem \eqref{eq:Od} (and hence \eqref{eq:SOd}) is solved exactly by its Lagrangian relaxation in the low noise regime.
  We point out that problem \eqref{eq:Od} has a very special structure,
  so its Lagrangian dual admits a more concise representation than a typical QCQP, see \cite[\S4.2]{Rosen2016}.
\end{example}

\begin{lemma}\label{lem:connectedgraph}
  Let $G=(V,E)$ be a connected graph, let $x_0\in \RR^k$, and let $L_{ij}:\RR^k\to \RR^k$ be invertible linear maps for $ij\in E$.
  Then the function $f(x):=\sum_{ij\in E}\|x_j - L_{ij}x_i\|^2$, where $x=(x_1,\dots,x_n)\in (\RR^k)^n$, is strictly convex.
\end{lemma}
\begin{proof}
  We may assume that the reference point $x_0 \!=\! 0$ after a possible affine transformation.
  Since ${f}(x)$ is convex and homogeneous, it suffices to see that ${f}(x)\!=\!0$ implies $x\!=\!0$.
  If ${f}(x)\!=\!0$ then $x_j\!=\!L_{ij}x_i$ for each $ij\!\in\! E$.
  Since ${x}_0\!=\!0$ and $G$ is connected it is clear that each $x_i$ must be zero.
\end{proof}

  An alternative QCQP formulation for the $SO(3)$ synchronization problem can be obtained by representing rotations with quaternions~\cite{Fredriksson2012}.
  The same analysis as above shows that the corresponding SDP relaxation is exact in the low noise regime, as was observed experimentally in~\cite{Fredriksson2012}.
  Tightness results for $SO(3)$ synchronization similar to ours were obtained in~\cite{Eriksson2018,Wang2013}.

\begin{example}[$SE(d)$ synchronization]
  A natural extension of the above problem is to replace rotation matrices by elements of the \emph{special Euclidean group} $SE(d)$.
  Given a graph $G=(V,E)$, and $\hat{R}_{ij}\in \RR^{d\times d}, \hat{u}_{ij} \in \RR^d$ for $ij\in E$, the problem is
  \begin{align}\label{eq:SEd}
    \min_{R_i\in SO(d),\, u_i \in \RR^d}\quad
    \sum_{ij\in E} \|R_j \!-\! \hat{R}_{ij}R_i\|_F^2
    + \|u_j \!-\! u_i \!-\! R_i \hat{u}_{ij}\|^2,
  \end{align}
  where $R_0 := \id_d, u_0 := 0$.
  As before, we can replace $SO(d)$ with $O(d)$ to obtain a QCQP, and consider its SDP relaxation.
  An argument similar to \Cref{lem:connectedgraph} shows that the objective function is strictly convex,
  and thus the Lagrangian relaxation solves problem~\eqref{eq:SEd} exactly under low noise.

  SDP relaxations for $SE(d)$ synchronization have received considerable attention in past years and similar exactness results have been derived \cite{Rosen2016,Wang2013}.
\end{example}

\begin{example}[Orthogonal Procrustes]
  Given $n,k,m_1,m_2 \!\in\! \NN$ and matrices
  $A \!\in\! \RR^{m_1\times n}$, $B\!\in\! \RR^{m_1\times m_2}$, $C\!\in\!\RR^{k\times m_2}$,
  the weighted orthogonal Procrustes problem, is
  \begin{align}\label{eq:procrustes}
    \min_{X\in \RR^{n\times k}}\quad &\|A X C - B\|_F^2,
    \quad \text{ s.t. }\quad X^\top X = \id_k.
  \end{align}
  The above is a QCQP parametrized by $\theta=(A,B,C)$.
  ACQ holds everywhere since the variety (the {\em Stiefel manifold}) is smooth and the ideal is radical.
  The objective function is strictly convex as long as the linear map $X \mapsto A X C$ is injective.
  In such cases \Cref{thm:firstresultinhom} guarantees that the SDP relaxation is exact under low noise.

  Problem~\eqref{eq:procrustes} may have several local optima, and thus local methods may fail~\cite{Viklands2006,Chu2001}.
  The above SDP relaxation was considered in~\cite{Cifuentes2015}.
\end{example}

\subsection{Estimation problems with degenerate objective}\label{s:estimationhard}

The previous applications had strictly convex objective functions,
and hence the theorems from \Cref{s:firstresult} were sufficient to analyze them.
The results from \Cref{s:main} can be used to study much more general estimation problems.
A preprint of this paper first appeared on the arXiv in 2017 and has since spawned a number of follow ups~\cite{Cifuentes2018,Zhao2019, Cifuentes2019}. For example \cite{Cifuentes2018,Zhao2019} used our \Cref{thm:main} to establish stability for a number of estimation problems from computer vision, control theory, and symbolic computation. As can be seen from these papers, verifying the assumptions of \Cref{thm:main} is non-trivial. So, we briefly mention a few of these applications (without proofs) 
to give the reader with a sense of how \Cref{thm:main} can be used. 
Several of these are nearest point problems of the form \eqref{eq:prob2} from the Introduction.


\begin{example}[Essential matrix estimation~\cite{Zhao2019}]
  A calibrated camera is described by a matrix of the form
  $P \!=\! (R \,|\, t)$, with $R \!\in\! SO(3)$, $t \!\in\! \RR^3$.
  Consider images $(u_j)_{j=1}^\ell \subset \RR^2$ and $(u_j')_{j=1}^\ell \subset \RR^2$ of the same 3D~points under two calibrated cameras~$P,P'$.
  The relationship between the image sets $(u_j)$ and $(u_j')$ is encoded by the essential matrix~$E $.
  This is a special $3 \!\times\! 3$ matrix that depends on $P,P'$.
  The problem of estimating an essential matrix $E$ based on noisy estimates~$(\hat{u}_j),(\hat{u}_j')$ of the images can be
  formulated as follows:
  \begin{align*}
    \min_{E\in\RR^{3\times 3},\, t\in\RR^3}\quad
    \sum_{j=1}^\ell\|\hat{v}_j^\top E\, \hat{v}_j'\|^2,
    \quad \text{ s.t. }\quad
    E E^\top = [t]_{\times} [t]_{\times}^\top\,,
    \quad
    t^\top t = 1,
  \end{align*}
  where $\hat{v}_j,\hat{v}_j' \in \RR^3$ are the homogenizations of
  $\hat{u}_j,\hat{u}_j'$, obtained by adding a last coordinate equal to one,
  and $[t]_{\times} \!\in \RR^{3 \times 3}$ is the skew-symmetric matrix, whose non-zero entries are
  $\pm t_i$, representing cross product with~$t$.
  The above is a homogeneous QCQP in variables $E,t$,
  parametrized by $\theta = (\hat{v}_j,\hat{v}_j')_{j=1}^\ell$.
  The objective is degenerate (its Hessian has corank~$3$) as it does not involve~$t$,
  so we cannot apply the results from \Cref{s:firstresult}.
  Nonetheless, it is shown in~\cite{Zhao2019} that the assumptions of \Cref{thm:main} hold,
  and hence this problem is solved exactly by the SDP relaxation under low noise.
\end{example}

The {sum-of-squares} (SOS) method provides a hierarchy of SDP relaxations of a QCQP, in which the first level is the Lagrangian relaxation.
Our methods can be used to analyze the SDP relaxations at any level,
as will be discussed in more detail in \Cref{s:sos}.
The remaining examples illustrations this.

\begin{example}[Approximate system realization~\cite{Cifuentes2018}]\label{ex:noisyhankel}
  Consider a discrete linear time invariant (LTI) system of order~$k$.
  Let $y=(y_1,y_2,y_3,\dots,y_n)$ be the truncated impulse response of the system,
  and assume that the signal $y$ is corrupted by Gaussian noise.
  The approximate realization problem consists of recovering the transfer function of the $k$-th order system from the corrupted signal~$\hat y$.
  The MLE can be found by solving the following least squares problem:
  \begin{align*}
    \min_{y\in \RR^n \!,\, z \in \RR^{k+1}} \;\|y - \hat{y}\|^2,
    \quad\text{ s.t. }\quad
    z^\top H_{k+1}(y) = 0,
  \end{align*}
  where $H_{k+1}(y)$ is the $(k{+}1) \!\times\! (n{-}k)$ Hankel matrix with the entries of~$y$.
  In particular, the transfer function of the system can be recovered from the optimal~$z$.
  The above problem is an inhomogeneous QCQP in variables $y,z$,
  parametrized by $\theta \!=\! \hat{y}$.
  Unfortunately, its Lagrangian relaxation does not provide any information about the value of the QCQP (c.f., \Cref{exmp:twistedcubicbad}).
  On the other hand,
  the second level of the SOS hierarchy always solves the problem exactly in the low noise regime.
  This result is proved in \cite{Cifuentes2018} by relying on \Cref{thm:main}
  (as the objective is degenerate, we cannot use \Cref{thm:firstresulthom}).
\end{example}

\begin{example}[Camera resectioning~\cite{Cifuentes2018}]
  In the uncalibrated resectioning problem we are given $\ell$~points $z_j\!\in\! \RR^4$
  and noisy images $\hat{u}_j\!\in\!\RR^2$ under an unknown camera $P \!\in\! \RR^{3\times 4}$,
  and the goal is to recover~$P$.
  Assuming i.i.d.\ Gaussian noise, the MLE is given by
  \begin{align*}
    \min_{P\in\RR^{3\times 4}, u_j\in\RR^2}\quad
    \sum_{j=1}^\ell\|u_j-\hat{u}_j\|^2,
    \quad \text{ s.t. }\quad
    u_j = \Pi P z_j\; \text{ for } j \in [\ell].
  \end{align*}
  where $\Pi:\RR^2\to \RR^2$, $(y_1,y_2,y_3) \mapsto (y_1/y_3,y_2/y_3)$.
  Though this problem looks similar to triangulation (\Cref{exmp:triangulation}),
  it is significantly harder because we cannot easily eliminate the auxiliary variables~$P$.
  The above problem can be formulated as a QCQP in the variables $P,u_j$ after clearing denominators.
  As before, the Lagrangian relaxation is non-informative.
  However, the SDP relaxation in the second level of the SOS hierarchy is always exact in the low noise regime,
  as shown in \cite{Cifuentes2018} by using \Cref{thm:main}.
\end{example}

\begin{example}[Approximate~GCD~\cite{Cifuentes2018}]
  Let $f_1 \!\in\! \RR[t]_{n_1}$, $f_2 \!\in\! \RR[t]_{n_2}$
  be univariate polynomials of degrees $n_1,n_2$,
  and let $g \!\in\! \RR[t]_{d}$ be their GCD, of degree~$d$.
  Here $\RR[t]_k$ denotes the vector space of univariate polynomials in $t$ of degree at most $k$.
  The approximate GCD problem consists of estimating the degree~$d$ polynomial $g$ from noisy estimates $\hat{f}_1,\hat{f}_2$.
  The MLE under Gaussian noise is:
  \begin{align*}
    \min_{f_1\in\RR[t]_{n_1},\,f_2\in\RR[t]_{n_2}, z\in \RR^{k}}\quad
    &\|f_1-\hat{f}_1\|^2 + \|f_2-\hat{f}_2\|^2,
    \quad \text{ s.t. }\quad
    z^\top \operatorname{Syl}_d(f_1,f_2) = 0,
  \end{align*}
  where $k := n_1 {+}n_2{-}2d$,
  and $\operatorname{Syl}_d(f_1,f_2)$ is the $k \times(k{+}d{-}1)$ Sylvester matrix,
  which is filled with the coefficients of $f_1,f_2$;
  see e.g.,~\cite{Kaltofen2006}.
  The GCD polynomial $g$ can be read from the vector~$z$.
  It is shown in \cite{Cifuentes2018}, using \Cref{thm:main},
  that the SDP relaxation in the second level of the SOS hierarchy is exact under low noise.
\end{example}

\section{SDP Stability in Polynomial Optimization}\label{s:sos}

Although the main focus of this paper has been the Lagrangian relaxation of a QCQP,
the techniques from this paper can be used in much greater generality.
In this section, we illustrate how to use our theorems to analyze SDP relaxations of polynomial optimization problems arising from the sum-of-squares (SOS) method.

We first consider unconstrained polynomial optimization.
Let $\RR[z]_{2d}$ be the vector space of multivariate polynomials of degree at most~$2d$ in variables $z=(z_1,\ldots,z_n)$.
Consider the parametric family of unconstrained polynomial optimization problems:
\begin{align}\label{eq:pop}
  \tag{$\mathit{POP}_{\theta}$}
  \min_{z \in \RR^n}\; p_\theta(z),
  \qquad\text{ where }
  p_\theta \in \RR[z]_{2d}
  \text{ depends continuously on }\theta.
\end{align}
We will analyze the stability of the SOS relaxation of $(\textit{{POP}\textsubscript{\texttheta}})$.

We briefly review the SOS method.
A polynomial $f\in\RR[z]_{2d}$ is SOS if it can be written in the form $f(z) = \sum_i f_i(z)^2$ for some $f_i\in\RR[z]_d$.
The set
\begin{align*}
  \Sigma_{n,2d} := \{ f \in \RR[z]_{2d} : f(z) \text{ is SOS }\}
\end{align*}
is a closed convex cone in $\RR[z]_{2d}$.
The \emph{SOS relaxation} of~\eqref{eq:pop} is
\begin{align}\label{eq:sos}
  \tag{$\mathit{SOS}_{\theta}$}
  \max_{\gamma \in \RR}\; \gamma,
  \quad \text{ s.t. }\quad
  p_\theta(z)-\gamma \in \Sigma_{n,2d}.
\end{align}
Since not all nonnegative polynomials are SOS, we have that $\textup{val}(\textit{{POP}\textsubscript{\texttheta}}) \geq
\textup{val}(\textit{{SOS}\textsubscript{\texttheta}})$.
The relaxation $(\textit{{SOS}\textsubscript{\texttheta}})$ can be solved
efficiently with an SDP,
and it is \emph{tight} at~$\theta$ if
$\val(\mathit{SOS}_\theta)=\val(\mathit{POP}_\theta)$.
If the minimizer of $p_\theta$ is unique, it might also be possible to recover it from the relaxation.

Assume that the SOS relaxation is tight for a specific parameter~$\bar\theta$.
We investigate the behavior of $(\mathit{SOS}_\theta)$ when $\theta$ is close to $\bar\theta$.
As the following example shows, stability is not to be taken for granted.

\begin{example}
  For the polynomial $p_\theta(z):= z_1^4z_2^2 + z_1^2z_2^4 + \theta z_1^2z_2^2\in \RR[z]_6 $ we have:
  \begin{align*}
  &\theta \geq 0 \quad\implies\quad
  \val(\mathit{POP}_\theta) = \val(\mathit{SOS}_\theta) = 0,
  \\
  &\theta < 0 \quad\implies\quad
  \val(\mathit{POP}_\theta) = \tfrac{1}{27}\theta^3
  \text{ and \eqref{eq:sos} is infeasible}.
  \end{align*}
  Hence the relaxation is not stable nearby $\bar\theta = 0$.
\end{example}

We will use \Cref{thm:firstresulthom} to establish a result that guarantees tightness of \eqref{eq:sos} near~$\bar \theta$.
The hypothesis for this theorem is a geometric condition
involving $p_{\bar \theta}$ and
the SOS cone $\Sigma_{n,2d}$, which we first explain.
Let $\bar z$ be the minimizer of $p_{\bar\theta}$.
Consider the following linear subspaces of $\RR[z]_{2d}$:
\begin{align*}
  H_{\bar z}:=\{ f \!\in\! \RR[z]_{2d}: f(\bar z) \!=\! 0 \},
  \qquad
  L_{\bar z}:=\{ f \!\in\! \RR[z]_{2d}: f(\bar z) \!=\! 0,\, \nabla f(\bar{z}) \!=\! 0\}.
\end{align*}
The subspace $H_{\bar z}$ is defined by a single linear equation,
so it is a hyperplane.
The intersections of the SOS cone with both subspaces agree.
Indeed, if $f(\bar z) \!=\! 0$ and $f$ is SOS,
then we must have that $\nabla f(\bar z) \!=\! 0$.
Let $K_{\bar z}$ be the (exposed) face of the SOS cone given by this
intersection:
\begin{align*}
  K_{\bar z}
  \,:=\, \Sigma_{n,2d} \cap H_{\bar z}
  \,=\, \Sigma_{n,2d} \cap L_{\bar z}.
\end{align*}
Observe that $p_{\bar \theta} \!-\! \bar{\gamma} \in K_{\bar z}$,
where $\bar\gamma := p_{\bar \theta}(\bar z)$ is the optimal value of $(\mathit{POP}_\theta)$.

\begin{theorem} \label{thm:sos}
  Let $\bar\theta$ be such that $\bar\gamma:=\val(\mathit{POP}_{\bar\theta})=\val(\mathit{SOS}_{\bar\theta})$ and there is a unique minimizer~$\bar z$.
  Consider the face $K_{\bar z}$ of the cone~$\Sigma_{n,2d}$, from above.
  If $p_{\bar\theta}-\bar\gamma$ lies in the relative interior of $K_{\bar z}$,
  then the relaxation \eqref{eq:sos} is tight and recovers the minimizer whenever $\theta$ is close enough to~$\bar\theta$.
\end{theorem}

\begin{proof}
  We may assume WLOG that $\bar\gamma \!=\! 0$,
  and thus $p_{\bar\theta}\in K_{\bar z}$.
  In order to use our methods, we need to rephrase~\eqref{eq:pop} as a QCQP.
  Let
  \begin{align*}
    x := ( z^\alpha)_{\alpha \in J} \in (\RR[z]_{d})^N,
    \quad\text{ where }\quad
    J := \{ \alpha \in \NN^n: \textstyle\sum_i \alpha_i \leq d\},
    \quad
    N := \binom{n+d}{d},
  \end{align*}
  be the vector with all monomials in $\RR[z]$ of degree at most~$d$.
  Note that any $p \!\in\! \RR[z]_{2d}$ can be written in the form
  $p(z) \!=\! x^\top F x$ for some $F \!\in\! \SR^N$.
  Such an $F$ is called a \emph{Gram matrix} of~$p$.
  Moreover, $p$ is SOS if and only if it has a positive semidefinite Gram matrix.

  Let $\bar{x}\!\in\!\RR^N$ be given by evaluating each of the monomials in $x$ at~$\bar{z}$. We first argue that if
  $p_{\bar\theta} \in \interior K_{\bar z}$, then
  \begin{enumerate}[label=(\roman*)]
    \item $p_{\bar\theta}$ has a Gram matrix $\bar{F}$ such that
      $\bar F \!\succeq\! 0$, $\bar{F} \bar{x}\!=\! 0 $, and $\corank \bar F \!=\! 1$.
    \item $F_\theta := \phi^\dagger (p_\theta\!-\!p_{\bar\theta}) + \bar{F}$
      is a Gram matrix of~$p_\theta$,
      where $\phi^\dagger$ is the pseudo-inverse of the linear map
      $ \phi: \SR^N \to \RR[z]_{2d}$,\, $A \mapsto x^\top A x$.
  \end{enumerate}

  Observe that $F$ is a Gram matrix of $p$ if and only if $\phi(F)\!=\!p$.
  Hence \emph{(ii)} follows by checking that $\phi(F_\theta) \!=\! p_\theta$.
  Also note that $K_{\bar z} \!=\! \phi(S)$, where
  $ S := \{F \!\in\! \SR^N: F \!\succeq\! 0, F\bar{x} \!=\! 0\}$.
  Indeed, $F \succeq 0$ if and only if $f(z) \!:=\! x^\top F x \!\in\! \Sigma_{n,2d}$, in which case $\nabla f(\bar z) \!=\! 2 F \bar x$,
  and hence $F \bar{x} \!=\! 0$ if and only if $\nabla f(\bar z)\!=\!0$.
  Since linear maps preserve relative interiors of convex sets,
  then $\interior K_{\bar z} = \phi(\interior S)$.
  Then \emph{(i)} follows by noticing that $\interior S = \{ F \!\in\! S: \corank F {=} 1\}$.

  The above properties provide a family of Gram matrices $F_\theta$ that depends continuously on $\theta$.
  Thus the parametric optimization problem~\eqref{eq:pop} can be phrased as
  \begin{align*}
    \min_{x \in \mathbf{X}} \; x^\top F_\theta x,
    \;\;\quad \text{where} \quad\;\;
    \mathbf{X} := \{ x\!=\!(z^\alpha)_{\alpha\in J}:  z\in \RR^n\} \subset \RR^N.
  \end{align*}
  The above is a QCQP since $\mathbf{X}$, the Veronese variety, is defined by quadratic equations:
  \begin{align*}
    \mathbf{X} := \{ x \in \RR^N :
      x_0 \!=\! 1,\;
      x_{\alpha_1} x_{\alpha_2} \!=\! x_{\beta_1} x_{\beta_2} \;\;
      \forall \alpha_1,\alpha_2,\beta_1,\beta_2 \in J
      \text{ s.t. } \alpha_1\!+\!\alpha_2 \!=\! \beta_1\!+\!\beta_2
    \}.
  \end{align*}
  Moreover, the relaxation~\eqref{eq:sos} coincides with the Lagrangian dual of the above QCQP.

  By construction we know that $F_{\bar\theta}\!\succeq\! 0$, $F_{\bar{\theta}}\bar{x}\!=\!0$, $\corank F_{\bar{\theta}}\!=\!1$.
  It is known that the Veronese variety is smooth and its ideal is radical.
  Hence ACQ holds everywhere by \Cref{thm:ACQradical}.
  The result now follows from \Cref{thm:firstresulthom},
\end{proof}

Our techniques can also be applied to constrained polynomial optimization, as we briefly explain.
Consider the parametric family of polynomial optimization problems:
\begin{equation}\label{eq:pop2}
  \tag{$\mathit{POP}_{\theta}^{\textup{con}}$}
  \begin{aligned}
    \min_{x\in\RR^N}\quad p_\theta(z),
    \qquad\text{s.t.}\qquad
    q_\theta^i(z) \!=\!0,\;\; i\!=\!1,\dots,m,
  \end{aligned}
\end{equation}
where $p_\theta,q_\theta^i \in \RR[z]_{2d}$
depend continuously on $\theta$.
Given $D \geq d$,
the $D$-th order {SOS relaxation} is:
\begin{align}\label{eq:sos2}
  \tag{$\mathit{SOS}_{\theta}^{\textup{con}}$}
  \max_{\gamma \in \RR,\; h^i \in \RR[z]}\; \gamma,
  \qquad\text{s.t.}\qquad
  p_\theta(z)-\gamma - \sum_i h^i(z)\, q^i_\theta(z) \,\in\, \Sigma_{n,2D},
\end{align}
where the optimization variables are the scalar $\gamma\!\in\!\RR$ and the polynomials $h^i \!\in\! \RR[z]_{2D-2d}$.
The relaxation~\eqref{eq:sos2} can be efficiently solved with an SDP.
We can use \Cref{thm:main} to analyze the stability of this relaxation.
In order to do that, we phrase \eqref{eq:pop2} as a QCQP, as before.
Namely, express the polynomials $p_\theta(z),q_\theta^i(z)$ as quadratic functions $f_\theta(x),g_\theta^i(x)$ in~$x$,
where $x \in \RR^{\binom{n+D}{D}}$ lies in the Veronese variety.

\section{Conclusion}

\Cref{thm:firstresulthom,thm:main} guarantee SDP-stability for
a parametrized family of QCQPs near a nominal parameter $\bar \theta$ at which
there is zero duality gap.
These results provide a uniform framework within which to understand and explain several previously observed occurrences of zero duality gap under low noise.
While the conditions of \Cref{thm:firstresulthom} are relatively easily to check,
\Cref{thm:main} requires technical assumptions that are more involved.
We believe that all the requirements in \Cref{thm:main} are necessary for stability.

Our results strongly depend on the quadratic nature of the problem. Although the conditions in  \Cref{thm:firstresultinhom,thm:main} make sense for general nonlinear programs, they only guarantee stability of zero duality gap in the QCQP case. For instance, the problem
$\min\{ x_1^2 \!+\! x_2^2 \!+\! \theta x_2^4 : x_2{=}0 \}$
satisfies the assumptions of \Cref{thm:firstresultinhom} for $\bar\theta{=}0$,
but the duality gap is not stable (it is zero only at~$\bar\theta$).

Our stability theorems not only guarantee zero duality gap nearby $\bar\theta$,
but also that the relaxation recovers the minimizer of the problem.
To do so, our theorems assume that the primal and dual optimal values are achieved at the nominal parameter~$\bar\theta$.
We leave for future work to investigate SDP stability in settings where the optimal values are not achieved.

As we have seen several times in this paper, polynomial optimization problems
can be formulated as QCQPs by adding extra variables. If these QCQPs satisfy the conditions of our theorems then they exhibit stability. As illustrated by \Cref{exmp:twistedcubic,exmp:twistedcubicbad}, different QCQP formulations of the same problem might have different stability properties.
It is natural to ask how to choose the {\em best} QCQP formulation. The SOS method gives a systematic procedure for constructing a hierarchy of QCQP formulations,
and it is optimal in the sense that it includes all valid relations up to a certain degree.
The Lagrangian relaxations studied in this paper correspond to the first SOS relaxation. However, the methods of this paper are general and
they can be used to study higher order SOS relaxations also.
In particular, \Cref{thm:main} was used in~\cite{Cifuentes2018} to analyze the second order SOS relaxation of some of the examples in Section~\ref{s:estimationhard}.

In this paper we have focused on equality-constrained QCQPs.
As is standard in nonlinear programming,
the results can be extended to account for inequality constraints
as long as the active constraints are preserved by the perturbation.
A comprehensive treatment of the inequality-constrained case is left for future work.

\subsection*{Acknowledgments}
We would like to thank Dmitriy Drusvyatskiy for many helpful conversations, and for suggesting the use of the implicit function theorem to analyze our problem.
Diego Cifuentes was in the Laboratory for Information and Decision Systems during the development of this paper.
Rekha Thomas was partially supported by the
NSF grant DMS-1719538.
This work was done in part while Pablo Parrilo was visiting the Simons Institute for the Theory of Computing.
It was partially supported by the DIMACS/Simons Collaboration on Bridging Continuous and Discrete Optimization through NSF Grant CCF-1740425.
This work was also supported in part by the Air Force Office of Scientific Research through AFOSR Grants FA9550-11-1-0305 and the National Science Foundation through NSF Grant CCF-1565235.

\appendix

\section{Stability of Lagrange multipliers}\label{s:implicitfunction}

The goal of this section is to show a generalization of \Cref{thm:aubin} to parametric nonlinear programming.
Let $\Theta \subset \RR^d$ be the parameter space,
let $f : \Theta\times\RR^N\to \RR$ and $g : \Theta\times\RR^N\to \RR^m$
be continuously differentiable,
and such that $f_\theta,g_\theta$ are twice continuously differentiable.
Consider the parametric family of nonlinear programs
\begin{align*}
  \min_{x\in \mathbf{X}_\theta} \quad f_\theta(x),
  \quad\text{ where }\quad
  \mathbf{X}_\theta := \{ x\in \RR^N: g_\theta(x) = 0 \}.
\end{align*}
Let $L_\theta(x,\lambda) := f_\theta(x) \!+\! \lambda^\top g_\theta(x)$ be the Lagrangian function, and let
\begin{equation*}
  \begin{aligned}
    \mathfrak{L}: \Theta \rightrightarrows \RR^N\times \RR^m,
    \qquad
    \theta\mapsto\,
    &\{(x,\lambda): g_\theta(x) \!=\! 0 ,\; \nabla_{x}L_\theta(x,\lambda) \!=\! 0\}
  \end{aligned}
\end{equation*}
be the Lagrange multiplier mapping.
Let $(\bar{x},\bar\lambda)\in \mathfrak{L}(\bar\theta)$ be a Lagrange multiplier pair at the nominal parameter~$\bar\theta$.
We denote $\bar{H} \!:=\! \frac{1}{2}\nabla^2_{xx} L_{\bar\theta}(\bar x,\bar\lambda)$.
We will derive conditions that ensure local stability of~$\mathfrak{L}$ nearby~$\bar\theta$.
The notion of stability we use is the \emph{Aubin property}; see~\cite{Dontchev2009,Rockafellar2009}.

\begin{definition}[Aubin property]
  Let $\mathfrak{F}: \RR^d \rightrightarrows \RR^n$ be a {set-valued mapping}.
  $\mathfrak{F}$~has the \emph{Aubin property}
  at $\bar{p} \!\in\! \RR^d$ for $\bar{y} \!\in\! \RR^n$
  if $\bar{y} \!\in\! \mathfrak{F}(\bar{p})$
  and there is a constant $\kappa \!\geq\! 0$
  and neighborhoods $U \!\ni\! \bar{y}, V \!\ni\! \bar{p}$ such that
  \begin{align*}
    \mathfrak{F}(p')\cap U \,\subset\,
    \mathfrak{F}(p) + \kappa\, |p'\!-\!p| \,\mathcal{B}
    \quad\mbox{ for all } p',p\in V,
  \end{align*}
  where $\mathcal{B} \subset \RR^n$ denotes the unit ball.
\end{definition}

The following is the main result of this section.

\begin{theorem}[Stability of Lagrange multipliers]\label{thm:aubin1}
  Let $(\bar{x},\bar\lambda)\in \mathfrak{L}(\bar\theta)$.
  Assume that $\ACQ{\bar{\mathbf{X}}}{\bar{x}}$ holds,
  the mapping $\theta \mapsto \mathbf{X}_\theta$ is smooth nearby $(\bar\theta,\bar x)$,
  and $v^\top \bar H v \neq 0$ for all nonzero $v \in T_{\bar x}(\mathbf{X}_{\bar\theta}) := \ker \nabla g_{\bar\theta}(\bar x)$.
  Then the mapping $\mathfrak{L}$ has the Aubin property at~$\bar\theta$ for $(\bar{x},\bar\lambda)$.
\end{theorem}

\begin{remark}
  Similar stability results about Lagrange multipliers appear in the literature (e.g.,~\cite{Bonnans2013,Fiacco1990}).
  However, we were not able to find a result that suited our needs.
  Previous results either have stronger assumptions (LICQ/MFCQ) or only imply outer semicontinuity of~$\mathfrak{L}$.
\end{remark}

The last assumption of \Cref{thm:aubin1} says that the quadratic form $v^\top \bar H v$ is nondegenerate on the tangent space $T_{\bar x}(\mathbf{X}_{\bar\theta})$.
This is similar, but weaker, to the second order sufficient condition for optimality,
which states that $v^\top \bar H v$ is strictly convex on $T_{\bar x}(\mathbf{X}_{\bar\theta})$.

Let us see that \Cref{thm:aubin1} implies \Cref{thm:aubin} from \Cref{s:main}.

\begin{lemma}\label{thm:continuouslocally}
  Let $\mathfrak{F}:\RR^d\rightrightarrows \RR^n$ be a mapping with closed graph.
  Assume that $\mathfrak{F}$ has the Aubin property at $\bar{p}$ for~$\bar{y}$.
  Then there exists a closed neighborhood $U_0\ni \bar{y}$ such that $p\mapsto \mathfrak{F}(p)\cap U_0$ is continuous at~$\bar{p}$.
\end{lemma}
\begin{proof}
  From the definition of the Aubin property it is clear that there exists a neighborhood $U_0\ni \bar{y}$ such that $\mathfrak{F}$ has the Aubin property at $\bar{p}$ for $y$, for any $y\in U_0 \cap \mathfrak{F}(\bar{p})$.
  We may assume that $U_0$ is closed.
  Let $\mathfrak{F}_0: p\mapsto \mathfrak{F}(p)\cap U_0$, and note that it has closed graph since $\mathfrak{F}$ does.
  Thus, $\mathfrak{F}_0$ is outer semicontinuous by \cite[Thm~5.7]{Rockafellar2009}.
  The lemma follows from \cite[Thm~9.38]{Rockafellar2009}.
\end{proof}

\begin{proof}[Proof of \Cref{thm:aubin}]
  Since $\bar H \succeq 0$,
  then $v^\top \bar H v \!=\! 0$ if and only if $\bar H v \!=\! 0$.
  Therefore, $v^\top \bar H v$ is nondegenerate on $T_{\bar x}(\mathbf{X}_{\bar\theta})$
  if and only if $\bar{x}$ is not a branch point of $\bar{\mathbf{X}}$ with respect to $v \mapsto \bar{H}v$.
  The proposition follows from \Cref{thm:aubin1,thm:continuouslocally}.
\end{proof}

We proceed to prove \Cref{thm:aubin1}.
The main technical tool we will use is the implicit function theorem, which can be phrased in terms of the Aubin property (see \cite[Ex.~4D.3]{Dontchev2009}).

\begin{theorem}[Implicit function]
  \label{thm:implicitfunction}
  Given $F : \RR^{d}\times\RR^{n}\to \RR^m$ continuously differentiable, let
  \begin{align*}
    \mathfrak{F}: \RR^d &\rightrightarrows \RR^n, \qquad
    p \mapsto \{y\in\RR^n: F(p,y)=0\}.
  \end{align*}
  Let $\bar{p},\bar{y}$ be such that $\bar{y}\in \mathfrak{F}(\bar{p})$.
  If $\nabla_y F (\bar{p},\bar{y})$ is surjective,
  then $\mathfrak{F}$ satisfies the Aubin property at $\bar{p}$ for $\bar{y}$.
\end{theorem}

\Cref{thm:aubin1} would be immediate if $\mathfrak{L}$ satisfied the hypothesis from \Cref{thm:implicitfunction}.
Unfortunately this is not true, since the defining equations of $\mathfrak{L}$ may have linearly dependent gradients.
In order to fix this problem, we consider a maximal subset of the equations $g'\subset g$ such that
$\{\nabla_x g^i_{\bar\theta}(\bar{x})\}_{g^i\in g'}$ are linearly independent.
Equivalently, $g'\subset g$ is such that $\nabla_x g'_{\bar{\theta}}(\bar{x})$ is full rank, and has the same rank as $\nabla_x g_{\bar{\theta}}(\bar{x})$.
Consider the modified solution mapping
\begin{gather*}
  \mathfrak{L}': \theta \mapsto
  \{(x,\lambda): {g}_\theta'(x)=0,\, \nabla_{x}L_\theta(x,\lambda) = 0 \}.
\end{gather*}
We now apply \Cref{thm:implicitfunction} to this new mapping.

\begin{lemma}\label{thm:aubinproperty}
  If $v^\top \bar H v \!\neq\! 0$ for all nonzero $v \!\in\! T_{\bar x}(\mathbf{X}_{\bar\theta})$,
  then $\mathfrak{L}'$ has the Aubin property at $\bar{\theta}$ for $(\bar{x},\bar\lambda)$.
\end{lemma}
\begin{proof}
  Let us see that the surjectivity condition from \Cref{thm:implicitfunction} is satisfied.
  To simplify the notation we will ignore the dependence on $\theta$, since the only parameter we consider in this proof is~$\bar\theta$.
  Let $J' \!:=\! \nabla_x g'(\bar{x})$, which is a submatrix of ${J} \!:=\! \nabla_x g(\bar{x})$.
  By construction, the rows of $J'$ are linearly independent and $\ker J' \!=\! \ker J$.
  Let $F(x,\lambda) := (g'(x), \nabla_x L(x,\lambda)$.
  We need to show that the rows of $\nabla F(\bar{x},\bar\lambda)$ are linearly independent.
  Observe that
  \begin{align*}
    \nabla_{\lambda,x} F(\bar{x},\bar\lambda) =
    \begin{pmatrix}
      0 & \nabla_x{g'}(\bar{x})\\
      \nabla_x{g}(\bar{x})^\top & \nabla^2_{xx} L(\bar x,\bar \lambda)
    \end{pmatrix}
    =
    \begin{pmatrix}
      0 & {J}'\\
      {J}^\top & 2\bar{H}
    \end{pmatrix}.
    \qquad
  \end{align*}
  Let $(u,v)$ be a vector in the left kernel of~$\nabla F(\bar{x},\bar\lambda)$,
  so that $v^\top J^\top \!=\! 0$, $u^\top J' {+} 2 v^\top \bar{H} \!=\! 0$.
  We need to show that $(u,v)\!=\!0$.
  Since $v \!\in\! \ker J \!=\! \ker J'$ then $0 = (u^\top J' {+} 2 v^\top \bar{H}) v  = 2 v^\top \bar{H} v$.
  As $v \in \ker J = T_{\bar x}(\bar{\mathbf{X}})$ and $\bar v^\top\bar{H} v \!=\! 0$, then $v \!=\! 0$ by the assumption.
  Therefore $0 = u^\top J' {+} 2 v^\top \bar{H} = u^\top J'$,
  and thus $u {=} 0$ since the rows of $J'$ are linearly independent.
\end{proof}

In order to prove \Cref{thm:aubin1} it remains to see that the modified mapping $\mathfrak{L}'$ agrees with $\mathfrak{L}$, at least locally.
This follows from the following lemma.

\begin{lemma} \label{thm:neighborhoods}
  Let $\mathbf{X}_\theta\subset {\mathbf{X}_\theta'}\subset \RR^N$ be the zero sets of $g_\theta, {g_\theta'}$.
  Assume that $\ACQ{\bar{\mathbf{X}}}{\bar{x}}$ holds,
  and that the mapping $\theta \mapsto \mathbf{X}_\theta$ is smooth nearby $(\bar\theta,\bar x)$.
  Then there are neighborhoods $V_0\ni \bar{\theta}$ and $U_0\ni \bar{x}$ such that
  $ \mathbf{X}_\theta \cap U_0 = \mathbf{X}_\theta' \cap U_0$
  for all $\theta\in V_0$.
\end{lemma}

The proof of \Cref{thm:neighborhoods} requires an auxiliary lemma.

\begin{lemma}\label{thm:locallyradical}
  Let
  $ \mathbf{W}:= \{w\!\in\!\RR^K: g(w){=}0\},$
  where $g=(g^1,\dots,g^m)$,
  and assume that $\mathbf{W}$ is a smooth $D$-dimensional manifold nearby~$\bar{w}$.
  Let $g'=(g^1,\ldots,g^{K-D})\subset g$ be such that their gradients at~$\bar{w}$ are linearly independent.
  Then there is a neighborhood $U\subset\RR^K$ of $\bar{w}$ such that $\mathbf{W}\cap U = \mathbf{W}'\cap U$,
  where $\mathbf{W}':=\{w: g'(w)=0\}$.
\end{lemma}
\begin{proof}
  By the implicit function theorem $\mathbf{W}'$ is a $D$-dimensional manifold nearby~$\bar{w}$.
  Thus, there is a neighborhood $U\subset \RR^K$ of $\bar{w}$ such that $\mathbf{W}\cap U$ is a submanifold of $\mathbf{W}'\cap U$.
  Since they have the same dimension, $\mathbf{W}\cap U$ must be an open set of $\mathbf{W}'\cap U$.
\end{proof}

\begin{proof}[Proof of \Cref{thm:locallyradical}]
  Let $\mathbf{W}:=\{(\theta,x): g_\theta(x){=}0\}$
  and $\mathbf{W}':=\{(\theta,x): g_\theta'(x){=}0\}$.
  We will use \Cref{thm:locallyradical} to show
  the existence of a neighborhood $U \!\ni\! \bar w$,
  such that $\mathbf{W}\cap U \!=\! \mathbf{W}'\cap U$.
  Note that this would conclude the proof.
  By assumption we know that $\mathbf{W}$ is a smooth manifold nearby~$\bar w:= (\bar{x},\bar\theta)$ of dimension $D:= \dim \Theta \!+\! \dim_{\bar{x}} \bar{\mathbf{X}} $.
  Recall that by construction of $g'$ the gradients $\{\nabla g^i_{\bar\theta}(\bar x)\}_{g^i \in g'}$ are linearly independent, and the number of equations is $|g'| = \rank \nabla g_{\bar\theta}(\bar x)$.
  Since $\ACQ{\bar{\mathbf{X}}}{\bar x}$ holds, then
  \begin{align*}
    |g'| = \rank \nabla g_{\bar\theta}(\bar x) = N - \dim_{\bar x} \bar{\mathbf{X}}
    = (\dim \Theta + N) - D.
  \end{align*}
  So the assumptions of \Cref{thm:locallyradical} are satisfied, as wanted.
\end{proof}

\begin{proof}[Proof of \Cref{thm:aubin1}]
  The Aubin property is a local condition.
  Since $\mathfrak{L},\mathfrak{L}'$ agree nearby $\bar{\theta},\bar{x}$ (\Cref{thm:neighborhoods}),
  and since $\mathfrak{L}'$ has the Aubin property (\Cref{thm:aubinproperty}), then the same holds for~$\mathfrak{L}$.
\end{proof}

\bibliographystyle{abbrv}
\bibliography{refs}

\begin{thebibliography}{10}

\bibitem{Aholt2012}
C.~Aholt, S.~Agarwal, and R.~Thomas.
\newblock A {QCQP} approach to triangulation.
\newblock In {\em ECCV (1)}, volume 7572 of {\em Lect. Notes Comput. Sci.},
  pages 654--667. Springer, 2012.

\bibitem{Aubin2009}
J.-P. Aubin and H.~Frankowska.
\newblock {\em Set-valued analysis}.
\newblock Springer Science \& Business Media, 2009.

\bibitem{Bazaraa2013}
M.~S. Bazaraa, H.~D. Sherali, and C.~M. Shetty.
\newblock {\em Nonlinear programming: theory and algorithms}.
\newblock John Wiley \& Sons, 2013.

\bibitem{Beck2006}
A.~Beck and Y.~C. Eldar.
\newblock Strong duality in nonconvex quadratic optimization with two quadratic
  constraints.
\newblock {\em SIAM J. Optim.}, 17(3):844--860, 2006.

\bibitem{Biswas2004}
P.~Biswas and Y.~Ye.
\newblock Semidefinite programming for ad hoc wireless sensor network
  localization.
\newblock In {\em Proc. Int. Symp. Inf. Process. Sens. Netw.}, pages 46--54.
  ACM, 2004.

\bibitem{Blekherman2013}
G.~Blekherman, P.~A. Parrilo, and R.~R. Thomas, editors.
\newblock {\em Semidefinite optimization and convex algebraic geometry},
  volume~13 of {\em Series Optim.}
\newblock MOS-SIAM, 2013.

\bibitem{Bochnak2013}
J.~Bochnak, M.~Coste, and M.-F. Roy.
\newblock {\em Real algebraic geometry}, volume~36.
\newblock Springer Science \& Business Media, 2013.

\bibitem{Bonnans2013}
J.~F. Bonnans and A.~Shapiro.
\newblock {\em Perturbation analysis of optimization problems}.
\newblock Springer Science \& Business Media, 2013.

\bibitem{Bookstein1979}
F.~L. Bookstein.
\newblock Fitting conic sections to scattered data.
\newblock {\em Computer graphics and image processing}, 9(1):56--71, 1979.

\bibitem{Chu2001}
M.~T. Chu and N.~T. Trendafilov.
\newblock The orthogonally constrained regression revisited.
\newblock {\em J. Comput. Graphical Stat.}, 10(4):746--771, 2001.

\bibitem{Cifuentes2018}
D.~Cifuentes.
\newblock A convex relaxation to compute the nearest structured rank deficient
  matrix.
\newblock {\em SIAM Journal on Matrix Analysis and Applications},
  42(2):708--729, 2021.

\bibitem{Cifuentes2019}
D.~Cifuentes, C.~Harris, and B.~Sturmfels.
\newblock The geometry of {SDP}-exactness in quadratic optimization.
\newblock {\em Math. Program.}, pages 1--30, 2019.

\bibitem{Cifuentes2015}
D.~Cifuentes and P.~A. Parrilo.
\newblock Sampling algebraic varieties for sum of squares programs.
\newblock {\em SIAM J. Optim.}, 27(4):2381--2404, 2017.

\bibitem{Dontchev2009}
A.~L. Dontchev and R.~T. Rockafellar.
\newblock {\em Implicit functions and solution mappings: a view from
  variational analysis}.
\newblock Springer Monographs Mathem. Springer, 2009.

\bibitem{Eriksson2018}
A.~Eriksson, C.~Olsson, F.~Kahl, and T.-J. Chin.
\newblock Rotation averaging and strong duality.
\newblock In {\em Proc. IEEE Conf. Comput. Vision Pattern Recognit.}, pages
  127--135, 2018.

\bibitem{Fazel2002}
M.~Fazel.
\newblock {\em Matrix rank minimization with applications}.
\newblock PhD thesis, Stanford University, 2002.

\bibitem{Fiacco1990}
A.~V. Fiacco and Y.~Ishizuka.
\newblock Sensitivity and stability analysis for nonlinear programming.
\newblock {\em Ann. Oper. Res.}, 27(1):215--235, 1990.

\bibitem{Finsler1936}
P.~Finsler.
\newblock {\"U}ber das {V}orkommen definiter und semidefiniter {F}ormen in
  {S}charen quadratischer {F}ormen.
\newblock {\em Comment. Math. Helv.}, 9(1):188--192, 1936.

\bibitem{Fredriksson2012}
J.~Fredriksson and C.~Olsson.
\newblock Simultaneous multiple rotation averaging using {L}agrangian duality.
\newblock In {\em Asian Conf. Comput. Vision}, pages 245--258. Springer, 2012.

\bibitem{Freund2004}
R.~W. Freund and F.~Jarre.
\newblock A sensitivity result for semidefinite programs.
\newblock {\em Oper. Res. Lett.}, 32(2):126--132, 2004.

\bibitem{Gouveia2010}
J.~Gouveia, P.~Parrilo, and R.~Thomas.
\newblock Theta bodies for polynomial ideals.
\newblock {\em SIAM J. Optim.}, 20(4):2097--2118, 2010.

\bibitem{Harris2013}
J.~Harris.
\newblock {\em Algebraic geometry: a first course}, volume 133.
\newblock Springer Science \& Business Media, 2013.

\bibitem{Heyden1997}
A.~Heyden and K.~{\AA}str{\"o}m.
\newblock Algebraic properties of multilinear constraints.
\newblock {\em Math. Methods Appl. Sci.}, 20(13):1135--1162, 1997.

\bibitem{Kaltofen2006}
E.~Kaltofen, Z.~Yang, and L.~Zhi.
\newblock Approximate greatest common divisors of several polynomials with
  linearly constrained coefficients and singular polynomials.
\newblock In {\em Proceedings of the 2006 international symposium on Symbolic
  and algebraic computation}, pages 169--176, 2006.

\bibitem{Kim2003}
S.~Kim and M.~Kojima.
\newblock Exact solutions of some nonconvex quadratic optimization problems via
  {SDP} and {SOCP} relaxations.
\newblock {\em Comput. Optim. Appl.}, 26(2):143--154, 2003.

\bibitem{Levy2000}
A.~B. Levy, R.~A. Poliquin, and R.~T. Rockafellar.
\newblock Stability of locally optimal solutions.
\newblock {\em SIAM J. Optim.}, 10(2):580--604, 2000.

\bibitem{Luo2010}
Z.-Q. Luo, W.-K. Ma, A.~M.-C. So, Y.~Ye, and S.~Zhang.
\newblock Semidefinite relaxation of quadratic optimization problems.
\newblock {\em IEEE Signal Process. Mag.}, 27(3):20--34, 2010.

\bibitem{Mordukhovich2013}
B.~S. Mordukhovich, R.~T. Rockafellar, and M.~E. Sarabi.
\newblock Characterizations of full stability in constrained optimization.
\newblock {\em SIAM J. Optim.}, 23(3):1810--1849, 2013.

\bibitem{Nayakkankuppam1999}
M.~V. Nayakkankuppam and M.~L. Overton.
\newblock Conditioning of semidefinite programs.
\newblock {\em Math. Program.}, 85(3):525--540, 1999.

\bibitem{Nie2014}
J.~Nie and L.~Wang.
\newblock Semidefinite relaxations for best rank-1 tensor approximations.
\newblock {\em SIAM J. Matrix Anal. Appl.}, 35(3):1155--1179, 2014.

\bibitem{Polik2007}
I.~P{\'o}lik and T.~Terlaky.
\newblock A survey of the {S}-lemma.
\newblock {\em SIAM Rev.}, 49(3):371--418, 2007.

\bibitem{Rockafellar2009}
R.~T. Rockafellar and R.~J.-B. Wets.
\newblock {\em Variational analysis}, volume 317.
\newblock Springer Science \& Business Media, 2009.

\bibitem{Rosen2016}
D.~Rosen, L.~Carlone, A.~Bandeira, and J.~Leonard.
\newblock A certifiably correct algorithm for synchronization over the special
  {E}uclidean group.
\newblock In {\em Intl. Workshop Algorithmic Found. Rob.}, 2016.

\bibitem{Stern1995}
R.~J. Stern and H.~Wolkowicz.
\newblock Indefinite trust region subproblems and nonsymmetric eigenvalue
  perturbations.
\newblock {\em SIAM J. Optim.}, 5(2):286--313, 1995.

\bibitem{Viklands2006}
T.~Viklands.
\newblock {\em Algorithms for the weighted orthogonal {P}rocrustes problem and
  other least squares problems}.
\newblock PhD thesis, Umea University, Sweden, 2006.

\bibitem{Waldspurger2015}
I.~Waldspurger, A.~d'Aspremont, and S.~Mallat.
\newblock Phase recovery, maxcut and complex semidefinite programming.
\newblock {\em Math. Program.}, 149(1-2):47--81, 2015.

\bibitem{Wang2013}
L.~Wang and A.~Singer.
\newblock Exact and stable recovery of rotations for robust synchronization.
\newblock {\em Information Inference}, 2(2):145--193, 2013.

\bibitem{Ye2003}
Y.~Ye and S.~Zhang.
\newblock New results on quadratic minimization.
\newblock {\em SIAM J. Optim.}, 14(1):245--267, 2003.

\bibitem{Zhang2000}
S.~Zhang.
\newblock Quadratic maximization and semidefinite relaxation.
\newblock {\em Math. Program.}, 87(3):453--465, 2000.

\bibitem{Zhao2019}
J.~Zhao.
\newblock An efficient solution to non-minimal case essential matrix
  estimation.
\newblock {\em IEEE Transactions on Pattern Analysis and Machine Intelligence},
  2020.

\bibitem{Zheng2012}
X.~Zheng, X.~Sun, D.~Li, and Y.~Xu.
\newblock On zero duality gap in nonconvex quadratic programming problems.
\newblock {\em J. Global Optim.}, 52(2):229--242, 2012.

\end{thebibliography}

\end{document}